\documentclass[a4paper]{amsart}
\usepackage{lmodern}
\usepackage{amsthm}
\usepackage{amsmath, amsthm}
\usepackage{amssymb}
\usepackage{setspace}
\usepackage{bbm}
\usepackage{comment}

\usepackage{hyperref}
\author[I.~Giron]{Itamar Giron}
\address[Itamar Giron]{Racah Institute of Physics \\
Edmond J. Safra Campus \\
The Hebrew University of Jerusalem \\
Givat Ram. Jerusalem, 9190401, Israel}
\email{itamar.giron@mail.huji.ac.il}
\thanks{Most of the results of this paper were part of the first author master thesis.}
\author[Y.~Hayut]{Yair Hayut}
\address[Yair Hayut]{Einstein Institute of Mathematics \\
Edmond J. Safra Campus \\
The Hebrew University of Jerusalem \\
Givat Ram. Jerusalem, 9190401, Israel}
\email{yair.hayut@mail.huji.ac.il}
\thanks{The second author was partially supported by the Israel Science Foundation grant 1967/21}

\newtheorem{theorem}{Theorem}[section]
\newtheorem{definition}[theorem]{Defintion}
\newtheorem{lemma}[theorem]{Lemma}
\newtheorem{claim}[theorem]{Claim}
\newtheorem{proposition}[theorem]{Proposition}
\newtheorem{corollary}[theorem]{Corollary}
\newtheorem{notation}[theorem]{Notation}
\newtheorem{remark}[theorem]{Remark}
\newtheorem{question}[theorem]{Question}
\newtheorem*{theorem*}{Theorem}
\DeclareMathOperator{\dom}{dom}
\DeclareMathOperator{\range}{range}
\DeclareMathOperator{\level}{Lev}
\DeclareMathOperator{\height}{height}
\DeclareMathOperator{\proj}{proj}
\DeclareMathOperator{\Lev}{Lev}
\newcommand{\CH}{\mathrm{CH}}

\begin{document}




\date{}
\title{Sealed Kurepa Trees}
\begin{abstract}
In this paper we investigate the problem of the distributivity of Kurepa trees. We show that it is consistent that there are Kurepa trees and for every Kurepa tree there is a small forcing notion which adds a branch to it without collapsing cardinals. On the other hand, we derive a proper forcing notion for making an arbitrary Kurepa tree into a non-distributive tree without collapsing $\aleph_1$ and $\aleph_2$. 
\end{abstract}
\maketitle


\section{Introduction}

In 1920, Suslin asked whether a Suslin line exists. While Aronszajn could not solve Suslin's problem, he did manage to construct a related linear order --- an Aronszajn line. In \cite{Kurepa}, both objects were translated to trees of height $\omega_1$: Aronszajn trees, which are $\omega_1$-trees without uncountable chains and Suslin trees, which are $\omega_1$-trees without uncountable chains or antichains. The tree that Aronjszajn essentially constructed was a \emph{special Aronszajn tree}, meaning that there is a function $f \colon T \to \omega$ which is injective on chains. This is a strong, and upwards absolute, witness for the lack of uncountable chains in the tree. 

In \cite{BaumgartnerMalitzReinhardt}, Baumgartner showed that there is a c.c.c.\ forcing that forces an arbitrary $\omega_1$-Aronszajn tree to be special. While this forcing can be generalized to an $\omega_1$-tree with at most $\aleph_1$ many branches (thus, makes the set of branches upwards absolute), there is no known parallel forcing that works for $\omega_1$-Kurepa trees. 

Our version for a special Kurepa tree is as follows.
\begin{definition}
    A Kurepa tree $T$ in a model $M$ is called \textbf{sealed} if $\forall\mathbb{P}\in M$ such that 
    \begin{equation*}
        \mathbb{P}\Vdash \exists \text{ a new branch for } T
    \end{equation*}

    Then $\mathbb{P}$ collapses $\omega_1$ or $\omega_2$.
\end{definition}
This definition intend to rule out possible forcing notions such as Woodin's stationary tower. While fully sealed Kurepa trees were constructed, see \cite{Poor} and \cite{HayutMuller}, it is open whether one can seal an arbitrary Kurepa tree using forcing, and the main result of the paper is a step towards this goal. 

This paper is organized as follows. In Section \ref{section:preliminaries} we provide some preliminaries and basic definitions that we will use. 

In Section \ref{section:distributive-trees}, we prove that assuming the existence of an inaccessible cardinal there exists a model in which for every Kurepa tree there exists a forcing-notion of cardinality $\aleph_1$ which adds a branch to it without collapsing $\aleph_1$ and $\aleph_2$. Moreover, in this model $\diamondsuit^+$ can hold.

The main result of this paper is a partial parallel of Baumgartner's theorem on specializing Aronszajn trees, that appears in Section \ref{section:sealing-trees-forcing}:
\begin{theorem*}
    Let $T$ be a Kurepa tree then there is a proper forcing notion that makes $T$ non-distributive, and preserves $\aleph_1$ and $\aleph_2$.
\end{theorem*}
Note that without preserving $\omega_2$, this result is rather trivial, as one can collapse the cardinality of the set of branches to $\aleph_1$ and use a variant of Baugartner's forcing in order to specialize it.  Getting a forcing notion that preserved $\aleph_2$ seems to be more challenging. Moreover, we prove that \emph{any} forcing notion from the ground model that adds a branch to $T$ will collapse $\omega_1$. Our proof uses Neeman's methods of two-type model sequence side conditions from \cite{Neeman}, and uses properness to prove that the cardinals are preserved. 

Section \ref{section:the-generic-object} contains the proofs that this forcing is well behaved and produces the correct combinatorial object. We conclude the paper with some open problems.

\section{Preliminaries}\label{section:preliminaries}
This section provides the necessary definitions, notations, lemmas and theorems we use in the proofs in this paper. 

For basic definitions related to trees, we refer the reader to \cite{Jech}. Our trees are always rooted (namely, contains a minimal element).







\begin{definition}
    Let $T$ be a tree and $x,y\in T$. The \underline{meet} of $x,y$ is the unique element $r\in T$ such that $r\leq_T x,y$
    and $\forall t\in T,$ if $r <_T t$ then $t\nleq x\vee t\nleq y$. 
    Assuming that $T$ is rooted, this element exists. We denote
    it by $x\wedge y$.
\end{definition}

\begin{lemma}
    \label{lem:wedge_properties}
    The following properties hold for the $\wedge$ operator:

    \begin{enumerate}
        \item \label{wedge_property_1} For $x,y\in T$ if $x\leq y$ then $x\wedge y = x$.
        \item \label{wedge_property_2} For $x,y,z\in T$ if $x\leq y$ then $z\wedge x\leq z\wedge y$
        \item \label{wedge_property_3} For $x,y,z \in T$ if $x\leq y$ and $x\nleq z$ then $x\wedge z = y\wedge z$.
        \item \label{wedge_property_idenpotency} For $x, y, z\in T$, then $(x \wedge y) \wedge z \in \{x \wedge y, x \wedge z\}$.
    \end{enumerate}
\end{lemma}
\begin{proof}
We will prove \ref{wedge_property_idenpotency}. The other properties are immediate.

First, the meet operator is associative and commutative. Since $x \wedge y, x\wedge z \leq x$, they are comparable. Without loss of generality, $x\wedge y \leq x \wedge z$ and in particular is it below $z$ and thus $(x\wedge y)\wedge z = x\wedge y$.
\end{proof}


\begin{definition}

    We denote the set of branches of $T$ by
    \begin{equation*}
        \left[T\right] \equiv \left\{ b\subset T\ \bigg|\ b\mathrm{\ is\ a\ branch}\right\}
    \end{equation*}
\end{definition}


\begin{theorem}[Baumgartner]
    Let $T$ be an $\omega_1$-Aronszajn tree, then there exists a c.c.c.\ forcing notion that adds a specializing function $f\colon T\to\omega$ using finite approximations.
\end{theorem}

\begin{definition}
    A tree $T$ is called a Kurepa tree if $T$ is an $\omega_{1}$-tree
    and $\left|\left[T\right]\right|=\aleph_{2}$.
\end{definition}

\begin{theorem}[Solovay \cite{JechSolovay}]
    \label{thm:solovay}
    If $V = L$, then there exists a Kurepa tree.
\end{theorem}

\begin{notation}
    A forcing notion is denoted by $\mathbb{P},\mathbb{Q},\mathbb{R}$ etc.
\end{notation}

\begin{notation}
    If $x,y\in \mathbb{P}$ and $\exists r\in \mathbb{P}$ such that $r\leq_{\mathbb{P}} x,y$ we denote $x\parallel y$. For a tree $T$ this translates to, $x, y\in T$ then $x\parallel y \iff x \leq_T y\vee y\leq_T x$.
\end{notation}

\begin{definition}
    Let $T$ be a tree, we define the tree-forcing $\mathbb{T} = T$ and $\leq_T\: =\: \geq_{\mathbb{T}}$ (reverse order).
\end{definition}

\begin{lemma}
    \label{lem:T_height_in_M}
    Let $T$ be an $\omega_1$-tree and let $M$ be a countable elementary submodel 
    of $H(\lambda)$ for some large regular cardinal $\lambda$, with $T\in M$. Then
    \begin{equation*}
        T\cap M = T\restriction M\cap\omega_1
    \end{equation*}
\end{lemma}

\begin{proof}
    $M \models T\ \mathrm{is\ an\ }\omega_1\mathrm{- tree}$. Thus, $T_\alpha \in M$ for all $\alpha < \omega_1\cap M$ since it is definable in $M$ using $T$. Since $M\models \left| T_\alpha\right| = \aleph_0$ there exists in $M$ a bijection $f:\omega\rightarrow T_\alpha$. Since $M$ is an elementary submodel $f$ is a bijection in the model. Since $\forall n < \omega, f(n)\in M$ and $f$ is a bijection it follows that $T_\alpha \subset M$.
\end{proof}

\begin{corollary}
    \label{cor:M_cap_omega1_geq}
    If $M,N$ are two countable elementary submodels of $H(\lambda)$ and $T\in M,N$ and $\exists x \in T\cap M \backslash N$ then $M\cap\omega_1 > N\cap\omega_1$.
\end{corollary}

\begin{theorem}[Shelah and Po{\'o}r \cite{Poor}]
    \label{thm:poor}
    Assume that there exists $A\subset \omega_1$ s.t.
    \begin{equation*}
        \omega_1^{L[A]} = \omega_1^V \wedge \omega_2^{L[A]} = \omega_2^V
    \end{equation*}
    then there exists a Kurepa tree with $[T]\subset L[A]$.
\end{theorem}

\begin{corollary}
    The Kurepa tree constructed in Theorem \ref{thm:poor} is sealed.
\end{corollary}

\begin{proof}
    Let $\mathbb{P}$ be some forcing poset and $G\subset \mathbb{P}$ a generic filter. Assume that

    \begin{equation*}
        \mathbb{P}\Vdash \omega_1^V = \omega_1^{V[G]}\wedge \mathbb{P}\Vdash \omega_2^V = \omega_2^{V[G]}
    \end{equation*}

    But then the construction of $L[A]$ is absolute thus $L[A]^V = L[A]^{V[G]}$. Since all the branches of $T$ are in $L[A]$ no branches are added.
\end{proof}

Let us recall the following combinatorial guessing principle.
\begin{definition}[Kunen]
A sequence $\langle \mathcal{A}_\alpha \mid \alpha < \omega_1\rangle$ is called a $\diamondsuit^+$-sequence if for every $\alpha < \omega_1$, $|\mathcal{A}_\alpha| \leq \aleph_0$ and for every $A \subseteq \omega_1$ there is a club $C \subseteq \omega_1$ such that for every $\alpha \in C$, $A\cap \alpha, C\cap \alpha \in \mathcal{A}_\alpha$.
\end{definition}

It was shown by Kunen that $\diamondsuit^+$ implies $\diamondsuit$, holds in $L$ (and generally, in fine structural inner models), and that $\diamondsuit^+$ implies the existence of a Kurepa tree. In fact the standard construction of $\diamondsuit^+$ in $L$ is a variation of Solovay's construction of a Kurepa tree.  

In Section \ref{section:distributive-trees}, we will show that it is consistent that $\diamond^+$ holds while no Kurepa tree is sealed.
\subsection{Properness and side conditions}
The definitions regarding properness are taken from \cite{Uri}.

\begin{definition}
    Let $\mathbb{P}$ be a forcing notion, $\lambda > 2^{\mathbb{P}}$ and $M\prec H(\lambda)$.

    A condition $p\in \mathbb{P}$ is called $M$-generic iff for every dense $D\subset \mathbb{P}$ such that $D\in M$ and for every $q\leq p$ there exists $d\parallel q$ such that $d\in D\cap M$.
\end{definition}

\begin{definition}
    A poset $\mathbb{P}$ is said to be proper for $M\prec H(\lambda)$ if $\forall p\in \mathbb{P}$, $\exists q\leq p$ such that $q$ is an $M$-generic condition.
\end{definition}

\begin{definition}
    For a set of models $\mathcal{M}$ we say $\mathbb{P}$ is proper for $\mathcal{M}$ if $\mathbb{P}$ is proper for all $M\in \mathcal{M}$.
\end{definition}

The main usage of properness is the preservation of cardinals (and more generally, preservation of stationarity of sets). Indeed, if there is a stationary set $\mathcal{M} \subseteq P_{\omega_1}(H(\lambda))$ for some large $\lambda$, such that $\mathbb{P}$ is proper for $\mathcal{M}$, then $\mathbb{P}$ does not collapse $\omega_1$.

If $\mathcal{M}$ is stationary in $P_{\omega_2}(H(\lambda))$ and consists of models that contain $\omega_1$, then properness for $\mathcal{M}$ implies the preservation of $\omega_2$.

Our intended forcing from Section \ref{section:sealing-trees-forcing} is designed to add a monotone function from a given Kurepa tree $T$ to $\omega$, in a way that will force it to be non-distributive. Adding this function using finite approximations in the naive way will end up in collapsing all relevant cardinals. So, instead we would like to limit to behaviour of those approximations in a way that will guarantee properness for some stationary set of models. This method, called \emph{forcing with side conditions}, was originated in the work of Todor\v{c}evi\'{c} for models of a single type, continued in various works for models of two types (see \cite{Mitchell, Krueger2, Krueger} and the formulation of Neeman that  we are using here, \cite{Neeman, Neeman2}). Later advances in the technique which were obtained by Velickovic and Mohhamadpour are likely to be crucial in extending this result.

Here we bring definitions and results from \cite{Neeman} that we use in this work. 

\begin{definition}
    \label{def:S_and_T}
    ${\mathcal S}$ and ${\mathcal T}$ are appropriate for cardinals $\kappa$,
    $\lambda$ and a transitive model $\left(K,\in\right)$ satisfying
    enough of \textbf{ZFC} and by enough we mean: closed under
    pairing, union, intersection, set difference, Cartesian product and
    transitive closure. Closed under range and restriction on functions
    and that for each $x\in K$ the closure of $x$ under intersections
    belongs to $K$, there is a bijection from an ordinal onto $x$ in
    $K$ and there is a sequence in $K$ consisting of members of $x$
    arranged in a non-decreasing Von-Neumann rank. If:
    \begin{enumerate}
        \item $\mathcal{T}$ is a collection of transitive $W\prec K$, and $\mathcal{S}$
              is a collection of $M\prec K$ with $\kappa\subset M$ and $\left|M\right|<\lambda$. $\forall N \in \mathcal{S}\cup\mathcal{T},\ N\in K,\ \left\{ \kappa,\lambda \right\}\in N$.

        \item $\forall M_{1},M_{2}\in{\mathcal S}$ and $M_{1}\in M_{2}$ then $M_{1}\subset M_{2}$.
        \item $W\in\mathcal{T}$ and $M\in\mathcal{S}$ and $W\in M$ then $M\cap W\in W$
              and $M\cap W\in\mathcal{S}$.
        \item Each $W\in\mathcal{T}$ is closed under sequences of length $<\kappa$.
    \end{enumerate}
\end{definition}

\begin{remark}
    Here, we concentrate on the definition \ref{def:S_and_T} in the case $\kappa =\omega$, $\lambda = \omega_1$ and $K = H(\omega_2)$. Thus, $\mathcal{S}$ becomes a collection of countable elementary submodels, and $\left( 2,4 \right)$ follow from elementarity. From now on we present only the discussed case (though we encourage reading about the general approach in \cite{Neeman}).
\end{remark}

\begin{definition}
    For $\mathcal{T},\mathcal{S}$ appropriate for $\omega,\omega_1,K$ we define the two-type model sequence poset $\mathbb{M} =\mathbb{M}_{\omega,\mathcal{S},\mathcal{T},K}$. $p\in\mathbb{M}$ iff:
    \begin{enumerate}
        \item $p = \langle  M_m\ \Big|\ m < n\rangle$ is a finite sequence that belongs to $K$ and $\forall m,\ M_m\in\mathcal{S}\cup\mathcal{T}$.
        \item The sequence is $\in$-increasing, $M_{m - 1} \in M_{m}$ for $0 < m < n$
        \item The sequence is closed under intersections, $\forall k,m < n,\ M_k\cap M_m$ is in $p$.
    \end{enumerate}
\end{definition}

\begin{definition}
    For $p,q\in\mathbb{M}$, where $p =\langle  M_\xi\ \Big|\ \xi <\gamma_p\rangle$ and $q =\langle  M_\xi\ \Big|\ \xi <\gamma_q\rangle$ we define $p\leq_\mathbb{M}q$ iff $\left\{ M_\xi\ \Big|\ \xi <\gamma_p \right\} \supset \left\{ M_\xi\ \Big|\ \xi <\gamma_q \right\}$.
\end{definition}

\begin{definition}
    Let $p\in\mathbb{M}$ and let $Q\in p$ define the residue of $p$ in $Q$.
    \begin{equation*}
        res_Q\left( p \right) = \left\{ M\in p\ \Big|\ M\in Q \right\}
    \end{equation*}
\end{definition}

\begin{lemma}
    \label{lem:why_d_M_is_a_condition}
    Let $p\in\mathbb{M}$ and $Q\in p$ then $res_Q\left( p \right)\in \mathbb{M}$.
\end{lemma}

\begin{definition}
    Conditions $p,q\in\mathbb{M}$ are compatible if $\exists r \in \mathbb{M}$ such that $r\supset p\cup q$ and are directly compatible if $r$ is the closure of $p\cup q$ under intersections.
\end{definition}

\begin{lemma}
    \label{lem:q_leq_res}
    Let $p\in \mathbb{M}$, $Q\in p$ and let $q\in Q\cap\mathbb{M}$ such that $q\leq res_Q\left( p \right)$ then $p,q$ are directly compatible.
\end{lemma}

\begin{corollary}
    \label{cor:p_cup_M}
    Let $M\in\mathcal{S}\cup\mathcal{T}$ and let $p\in M\cap\mathbb{M}$ then $\exists q \leq p$ such that $M\in q$. Moreover, $q$ can be taken to be the closure of $p\cup \left\{ M \right\}$ under intersections.
\end{corollary}
    
\section{Every Kurepa tree is not-sealed}\label{section:distributive-trees}

Our aim in this section is to construct a forcing that gives us a model in which every Kurepa tree is not sealed. In the first part we deal with Solovay's tree in $L$ and show that it is distributive there. In the second part, we assume the existence of an inaccessible cardinal, and construct a model in which there are Kurepa trees (and even $\diamondsuit^+$ holds) but for every Kurepa tree there is a forcing notion of size $\aleph_1$ that introduces a new branch.

\subsection{Distributivity of Solovay's Tree}
In this subsection we show that the Kurepa tree that Solovay constructed in $L$ is distributive in $L$. This implies that the result of Section \ref{section:sealing-trees-forcing} is non vacuous. 

\begin{theorem}\label{thm:solovay-tree-is-distributive} Assuming $V = L$, $\mathbb{T}$ adds a branch to Solovay's Kurepa tree (the tree which was declared in Theorem \ref{thm:solovay}) without collapsing $\omega_1$ or $\omega_2$.
\end{theorem}

\begin{proof}
    We present briefly the construction of the tree (without proof that it is Kurepa). Define a function

    \begin{equation*}
        f(\alpha) = \min \left\{ \gamma\, \bigg|\, \alpha \in L_\gamma \prec L_{\omega_1} \right\}
    \end{equation*}

    Using this $f$ we define a Kurepa family (namely, a family of size $\aleph_2$ with countable restrictions).

    \begin{equation*}
        \mathcal{F} = \left\{ X\subset \omega_1\, \bigg|\, \forall
        \alpha <\omega_1, X\cap \alpha \in L_{f(\alpha)} \right\}
    \end{equation*}

    Using this family we define a Kurepa tree constructed from the functions $g_X(\alpha) = X\cap\alpha$

    \begin{equation*}
        T = \left\{ g_X\restriction\beta\, \bigg|\, \beta <\omega_1 \right\}
    \end{equation*}

    Let $\mathbb{T}$ be the tree forcing, we want to show $\mathbb{T}$ is $\omega_1$-distributive. Let $\vec{D} = \langle  D_n\: \bigg|\: n <\omega\rangle$ be some collection of open dense sets and let $M\prec L_{\omega_2}$ be a countable elementary submodel such that $T, \vec{D}\in M$. Let $\pi: M\rightarrow \overline{M}$ be the transitive collapse and suppose $M\cap\omega_1 = \gamma$. Note that $\mathrm{crit}\ \pi^{-1} = \gamma$, so for every $t\in T \restriction \gamma$, $t\in M$ and $\pi(t) = t$.\footnote{The notation $\mathrm{crit }\pi^{-1}$ refers to the \emph{critical point} of $\pi$. The critical point of an elementary embedding between (transitive) models of fraction of set theory is the least ordinal that moves.}
    
    Since

    \begin{equation*}
        M\models D_n \text{ are open dense} \Longrightarrow \overline{M}\models \pi(D_n)\text{ are open dense}
    \end{equation*}

    by \ref{lem:T_height_in_M} $\pi(T) = T\restriction\gamma$ and from the condensation lemma

    \begin{equation*}
        \exists\zeta<\omega_1,\: \overline{M} = L_\zeta\models\gamma = \omega_1\wedge V = L
    \end{equation*}

    Since $\gamma \in L_\zeta$, we claim that $L_\zeta\in L_{f(\gamma)}$. Indeed, since $L_{f(\gamma)}\prec L_{\omega_1}$, $\gamma \in L_{f(\gamma)}$ and $L_{\omega_1}\models |\gamma| = \aleph_0$ we have by elementarity that $L_{f(\gamma)} \models |\gamma|=\aleph_0$. In particular $f(\gamma) > \zeta$, as in $L_\zeta$, $\gamma$ is uncountable and thus the function witnessing $|\gamma|=\aleph_0$ cannot be a member of $L_{\zeta}$ and must be constructed at some level between $\zeta$ and $f(\gamma)$. 
    
    Let $t_0\in L_{f(\gamma)}\cap \mathbb{T}$. Take $\overline{t}_{n + 1}\in \pi(D_n)$ be the first in the $\leq_L$ ordering such that $\overline{t}_{n + 1} \leq_{\mathbb T} \overline{t}_n$. Define $t_n = \pi^{ -1}(\overline{t}_n)$ then $t_{n + 1}\in D_n$. Since $\pi(\vec{D})\in \overline{M}\subset L_{f(\gamma)}$ the sequence $\langle  \overline{t}_n\, \big|\, n <\omega\rangle\in \overline{M}$. Remember that $\pi(t) = t$, thus $\langle  t_n\, \big|\, n <\omega\rangle\in \overline{M}$, denote $t_n = g_{X_n}\restriction\alpha_n$ and consider

    \begin{equation*}
        \bigcup g_{X_n}\restriction \alpha_n = g_{\bigcup X_n\cap \alpha_n}
    \end{equation*}

    We want to show that $g_{\bigcup X_n\cap \alpha_n}\in L_{f(\gamma)}$. Since $\mathrm{cf}\, \gamma = \omega$ and $\pi(D_n)$ are open dense we can impose another restriction on $t_n$, and assume that they satisfy $\sup \alpha_n = \gamma$. Let $\beta < \gamma$ consider $\left(\bigcup X_n\cap \alpha_n\right)\cap \beta = (\bigcup X_n) \cap \beta $. By construction
    \begin{equation*}
        \forall n <\omega,\, X_n\cap\beta \in L_{f(\beta )}\wedge L_{f(\beta)}\prec L_{\omega_1}
    \end{equation*}

    Thus, $\langle  X_n\cap \beta\: \big|\: n <\omega\rangle\in L_{f(\beta)}$ therefore $\bigcup X_n\cap\beta\in L_f(\beta)$ and $\bigcup X_n \cap \alpha_n \in \mathcal{F}$.

    Since $\mathbb{T}$ is distributive it does not add any reals and in particular does not collapse $\omega_1$.

    As $|\mathbb{T}| = \aleph_1$, it is trivially $\aleph_2$-c.c.
\end{proof}

\subsection{A model without sealed trees}
In this subsection we intend to obtain two slightly different models in which for every Kurepa tree $T$ there is a $\sigma$-distributive forcing notion $\mathbb{Q}$ of cardinality $\aleph_1$ that adds a branch through $T$. By the result of Shelah and Po{\'o}r, \cite{Poor}, an inaccessible cardinal is required. 

Our first model is going to be constructed by first collapsing an inaccessible cardinal to be $\aleph_2$ and then introducing a generic Kurepa tree, $T$. In this model, there is a forcing notion which is a countable support product of length $\omega_1$, of adding a generic branch to $T$, that introduces a new branch to any Kurepa tree. 

In our second variant, instead of adding a generic Kurepa tree we will add a generic $\diamondsuit^+$-sequence. In this variant, the forcing for introducing a new branch to a Kurepa tree depends on the Kurepa tree.

The following forcing notion was introduces by Jech.
\begin{definition}
    \label{def:K_forcing_notion}
    Let $\mathbb{K}_\kappa$ be the following forcing notion. 
    
    $p\in \mathbb{K}_\kappa$ if and only if $p = \langle  T_p, b_p \rangle$ where $T_p$ is an initial segment of an $\omega_1$-tree with a maximal level and $b\colon \kappa\rightarrow T_{\max}$, where $T_{\max}$ is the maximal level of $T_p$ and $|b|<\aleph_1$. 
    
    Define the partial order $\leq_{\mathbb{K}_{\kappa}}$ such that $p\leq_{\mathbb{K}_\kappa} q$ if and only if $T_q$ is an initial segment of $T_p$ (which we denote by $T_q \trianglelefteq T_p$), $\dom b_q\subset\dom b_p$ and $\forall \alpha \in \dom b_q,\, b_q(\alpha)\leq_{T_p} b_p(\alpha)$.
\end{definition}

Jech showed that $\mathbb{K}_{\kappa}$ is $\sigma$-closed and assuming $\CH$ it is $\lambda$-Knaster for every regular $\lambda$ such that $\aleph_2 \leq \lambda \leq \kappa$. Thus, assuming \textbf{CH}, $\mathbb{K}_{\kappa}$ adds a Kurepa tree with $\kappa$ many branches without collapsing cardinals. 

In the construction ahead we will use the definition of $\mathbb{K}_\kappa$ for either $\kappa = \aleph_2$ or $\kappa$ inaccessible cardinal (which will be collapsed to $\aleph_2$). 

As we will need some specific details about the $\sigma$-closure of the forcing, we include here the standard proof. Let us fix now a regular cardinal $\kappa \geq \omega_2$.
\begin{lemma}
    \label{lem:K_sigma_closed}
    $\mathbb{K}_\kappa$ is $\sigma$-closed.
\end{lemma}

\begin{proof}
    Let $\langle  p_n \mid n <\omega \rangle$ be a sequence of decreasing conditions in $\mathbb{K}_\kappa$. Let $p_n =\langle  T_n, b_n\rangle$, and define $T_\omega = \bigcup_{n<\omega}T_n$. We would like to verify that $T_{\omega}$, maybe after adding some nodes, can be the tree part of a condition. 
    
    Since $T_n$ is an initial segment of $T_{n + 1}$ we get that $T_\omega$ is a tree. Moreover, as for every $n < \omega$, $\height(T_n) <\omega_1$, the height of $T_\omega$ is $\sup_{n < \omega} \height(T_n)$ which is countable as well.
    
We separate into two cases. If $T_\omega$ has a maximal level (i.e.\ $\height(T_\omega)$ is a successor ordinal), then there must be $n < \omega$ with $\height(T_n) = \height(T_\omega)$. Take
   \begin{equation*}
        b_\omega =\bigcup_{\max T_n = \max T_\omega} b_n,
    \end{equation*}

    and define the condition $p_\omega =  \langle  T_\omega, b_\omega \rangle$. This condition satisfies $\forall n < \omega,\, p_\omega \leq p_n$.

    If $T_\omega$ does not have a maximal level, we must add one to it in order to obtain a condition. Let us build this additional top level, using the branch functions $b_n$. 
    
    Consider $X = \bigcup_{n< \omega} \dom b_n$.     
    For every ordinal  $\alpha \in X$ we can look at

    \begin{equation*}
        B_\alpha = \left \{ b_n(\alpha) \mid n < \omega\wedge \alpha\in \dom b_n\right \}
    \end{equation*}

    Define $t_\alpha$ to be the element in the level $T_{\max}$ above the branch generated by $B_\alpha$. Explicitly, we define $t_\alpha$ s.t.\ $\forall x\in B_\alpha,\, x\leq t_\alpha$ and $T_{\max} = \left\{ t_\alpha\mid\alpha\in X \right\}$. We then define $T_{\omega + 1} = T_\omega \cup T_{\max}$ and $b_{\omega + 1} \colon X\rightarrow T_{\max}$ by $b_{\omega +1}(\alpha) = t_\alpha$. This $p_{\omega +1} = \langle T_{\omega +1}, b_{\omega + 1} \rangle $ is a condition in $\mathbb{K}_\kappa$ and satisfies $\forall n <\omega,\, p_{\omega +1} \leq p_n$.
\end{proof}

Let $G \subseteq \mathbb{K}_{\kappa}$ be a generic filter. Then, in $V[G]$ there is a tree $T = \bigcup_{\langle  t, b\rangle \in G} t$, and $B \subseteq T$ is a branch if and only if there is $\alpha$ such that $B = \{b(\alpha) \mid \langle  t,b \rangle \in G, \alpha \in \dom b\}$.
\begin{definition}
    \label{def:K_rest_delta}
    Let $\omega_1\leq \delta <\kappa$ be some ordinal, we define $\mathbb{K}_\delta$ to be

    \begin{equation*}
        \mathbb{K}_\delta = \left\{ r\in \mathbb{K}_\kappa\, \Big|\, \dom b_r\subseteq\delta \right\}
    \end{equation*}

    And we define $\mathbb{K}_{[\delta,\kappa)}$ in the generic extension by $\mathbb{K}_\delta$ to be
    \begin{equation*}
        \mathbb{K}_{[\delta,\kappa)} = \left\{ b_r \, \Big|\, r\in \mathbb{K}_\kappa, \dom \: b_r \subseteq [\delta, \kappa)\right\}
    \end{equation*}
\end{definition}

\begin{remark}
    Notice that the conditions in $\mathbb{K}_{[\delta,\kappa)}$ consists only of the branch function part. This is because $\mathbb{K}_{[\delta,\kappa)}$ is defined in the generic extension by $\mathbb{K}_\delta$ in which the tree part will be entirely determined.
\end{remark}

\begin{lemma}
    \label{lem:K_sim_K_delta}
    $\mathbb{K}_\kappa \simeq \left(\mathbb{K}_\delta\right) \ast \dot{\mathbb{K}}_{[\delta, \kappa)}$.
\end{lemma}

\begin{proof}
    We build a projection map $\pi \colon \mathbb{K}_\kappa \rightarrow \mathbb{K}_\delta$

    \begin{equation*}
        \pi \left(\langle  T, b\rangle\right) = \langle  T, b \restriction \delta\rangle
    \end{equation*}

    We show it is indeed a projection, let $p,q\in \mathbb{K}_\kappa$ be conditions such that $p\leq q$. By definition, $T_q \triangleleft T_p$, $\dom b_q\subset\dom b_p$ and $\forall \alpha\in \dom b_p,\, b_q(\alpha) = b_p(\alpha)$. Consider $\pi(p),\pi(q)$ since the tree part does not change and  $\delta\cap\dom b_q\subset\delta\cap\dom b_p$, we have $\pi(p)\leq\pi(q)$.

    Let $p\in \mathbb{K}_\kappa$ and let $q\leq \pi(p)$, consider $\langle  T_q, b_q\cup b_p^\star\rangle $ (where $b_p^\star$ sends $\delta \leq  \alpha$ to some element in the maximal level of $T_q$ over the respective $b_p(\alpha)$ in the maximal level of $T_p$) since $b_q\in \mathbb{K}_\delta$ this is a condition in $\mathbb{K}_\kappa$ and $\pi(\langle  T_q, b_q\cup b_p^\star\rangle) = q\leq q$. Hence, $\pi$ is a projection.

    Let $H\subset \mathbb{K}_\delta$ be a generic filter and let

    \begin{equation*}
        \mathbb{Q} = \left\{ p\in \mathbb{K}_\kappa\, \Big|\, \pi(p)\in H \right\}
    \end{equation*}

    Denote by $\dot{\mathbb{Q}}$ the $\mathbb{K}_\delta$-name for $\mathbb{Q}$. We define an embedding $\iota\colon\mathbb{K}_\kappa\rightarrow (\mathbb{K}_\delta) \ast \dot{\mathbb{Q}}$ via $\iota(p) = (\pi(p),\dot{p})$ where $\mathbbm{1}_{\mathbb{K}_\delta}\Vdash_{\mathbb{K}_\delta} \dot{p}\in \dot{\mathbb{Q}}$ and $\pi(p)\Vdash_{\mathbb{K}_\delta} \dot{p} = \check{p}$. This is well known to be a dense embedding. Therefore, $\mathbb{K}_\kappa \simeq \mathbb{K}_\delta \ast \dot{\mathbb{Q}}$. We define another embedding $\iota' \colon \iota(\mathbb{K})\rightarrow\mathbb{K}_\delta \ast \dot{\mathbb{K}}_{[\delta, \kappa)}$ via $\iota'(\pi(p), \check{p}) = (\pi(p), \check{b_p}\restriction [\delta, \kappa))$\footnote{This definition applies for a dense subset of $\iota(\mathbb{K})$, and can be extended to the whole of $\iota(\mathbb{K})$.}. 
\end{proof}

\begin{lemma}
    $\iota'\colon \iota\left(\mathbb{K}_\kappa\right)\rightarrow\mathbb{K}_\delta \ast \dot{\mathbb{K}}_{[\delta, \kappa)}$ is a dense embedding.
\end{lemma}

\begin{proof}
    Let $r = \langle \langle  T, b\rangle, \dot{b}_\star\rangle\in \mathbb{K}_\delta \ast \dot{\mathbb{K}}_{[\delta, \kappa)}$ be some condition. We want to show that there exists a condition in $\iota(\mathbb{K}\ast \dot{\mathbb{Q}})$ whose image is below $r$. First choose a condition $ \langle  T', b' \rangle \leq \langle  T,b\rangle$ which determines $\dot{b}_\star$ i.e. $\langle  T', b'\rangle\Vdash_{\mathbb{K}_{\delta}} \dot{b}_\star = \check{b}_\star$ for some $\mathbb{K}_\delta$-name $\check{b}_\star$ of a branch function.

    Now consider $(\pi(p), \check{p})$ where $p$ is the condition $(T, b \cup b_\star)\in \mathbb{K}_\kappa$. By definition $\pi(p) =\langle T, b\rangle$ since $\dom b\subset \delta$ and $\dom \, b'\in [\delta, \kappa)$. Thus,
    \begin{equation*}
        \iota'(\pi(p), \check{p}) =\langle \pi(p), \check{b}_\star\rangle = \langle  \langle  T', b' \rangle, \check{b}_\star \rangle\leq \langle  \langle  T,b\rangle , \dot{b}_\star \rangle = r.
    \end{equation*}
\end{proof}

\begin{proof}[Back to the proof of lemma \ref{lem:K_sim_K_delta}]
    Since, $\iota, \iota'$ are both dense embeddings (and thus, also their composition) we get the equivalence

    \begin{equation*}
        \mathbb{K}_\kappa \simeq \mathbb{K}_\delta \ast \dot{\mathbb{Q}} \simeq \mathbb{K}_\delta \ast \dot{\mathbb{K}}_{[\delta,\kappa)}
    \end{equation*}
\end{proof}

\begin{theorem}
    \label{thm:every_Kurepa_not_sealed}
    It is consistent, relative to the existence  of a strongly inaccessible cardinal, that there exists a $\sigma$-distributive forcing-notion $\mathbb R$ of cardinality $\aleph_1$, such that for every Kurepa tree $T$ in the generic extension, $\mathbb{R}$ adds a branch to $T$. In particular, $T$ is not sealed.
\end{theorem}

\begin{proof}
    \renewcommand{\qedsymbol}{}
    Let $\kappa$ be an inaccessible cardinal. Consider the forcing notion $\mathbb{CK}_\kappa\equiv\mathbb{C}_{\omega_1, <\kappa}\times \mathbb{K_\kappa}$, where $\mathbb{C}_{\omega_1, < \kappa}$ is the Levy collapse\footnote{The Levy collapse is the collection of all partial functions $f \colon \kappa \times \omega_1 \to \kappa$ such that $\dom f$ is at most countable and for every $0 < \mu < \kappa$ and $\alpha < \omega_1$ such that $(\mu,\alpha)\in\dom f$, $f(\mu,\alpha) < \mu$. We order it by reverse inclusion.}.
\end{proof}

\begin{lemma}
    $\mathbb{CK}_\kappa$  is $\sigma$-closed.
\end{lemma}

\begin{proof}
    Let $\langle \left( p_n,q_n \right)\in\mathbb{CK}_\kappa\mid n <\omega \rangle $ be a decreasing sequence of conditions. Consider the condition $\left(T_{\omega + 1}, b_{\omega + 1}\right)$ from lemma \ref{lem:K_sigma_closed} and the function $p_\omega = \bigcup_{n < \omega}p_n$.

    Then, $\left( p_\omega, \langle T_{\omega + 1}, b_{\omega + 1}\rangle \right)$ is a condition and $\left( p_\omega, \langle T_{\omega + 1}, b_{\omega + 1}\rangle \right) \leq (p_n,q_n)$ for all $n < \omega$.
\end{proof}

\begin{lemma}
    \label{lem:kappa_c_c}
    $\mathbb{CK}_\kappa$ is $\kappa$-c.c.
\end{lemma}

\begin{proof}
	Indeed, $\mathbb{CK}_{\kappa}$ is equivalent to a \emph{product} of two $\kappa$-Knaster forcing notions (see \cite[Chapter 5]{CummingsHandbook}), and thus it is $\kappa$-Knaster and in particular, $\kappa$-c.c. 
\end{proof}	

\begin{definition}
    Let $\delta <\kappa$ be some ordinal. We define the following posets
    \begin{equation*}
        \mathbb{C}_{\omega_1, \delta} = \left\{ r\in   \mathbb{C}_{\omega_1, <\kappa}\,  \Big|\, \dom r \subset \delta \times \omega_1\right\}
    \end{equation*}
    and
    \begin{equation*}
        \mathbb{C}_{[\delta, \kappa )} = \left\{ r\in \mathbb{C}_{\omega_1, <\kappa}\, \Big|\, \dom r \subset [\delta, \kappa )\times \omega_1 \right\}
    \end{equation*}
    And together with definition \ref{def:K_rest_delta} we build the posets
    \begin{equation*}
        \mathbb{CK}_\delta = \mathbb{C}_{\omega_1, \delta}\times \mathbb{K}_\delta
    \end{equation*}
    And
    \begin{equation*}
        \mathbb{CK}_{[\delta, \kappa)} = \mathbb{C}_{[\delta, \kappa )}\times \mathbb{K}_{[\delta, \kappa)}
    \end{equation*}

\end{definition}

\begin{proof}[Back to the proof of Theorem \ref{thm:every_Kurepa_not_sealed}]

    Let $\dot{T}$ be a name s.t.\ $\dot{T}$ is forced to be a Kurepa tree in the generic extension by $\mathbb{CK}_{\kappa}$. Clearly, there is such a name, as the generic tree added by $\mathbb{K}_{\kappa}$ is Kurepa.

Let
    \begin{equation*}
        \dot{T} = \left\{ \langle  r, \left( \check{\alpha}, \check{\beta} \right)\rangle\, \Big|\, r\in A_{\alpha, \beta},\, \alpha,\beta <\omega_1 \right\}
    \end{equation*}

    be a nice name for $T$ (for the definition of nice names and their properties see \cite{Kunen}). Since $A_{\alpha,\beta}$ is an anti-chain and $\mathbb{CK}_\kappa$ is $\kappa$-c.c.\ we have that $|A_{\alpha,\beta}|<\kappa$. Let
    \begin{equation*}
        \delta = \sup \left\{ \pi_1(\dom f_r)\cup\dom b_r\, \Big|\, \alpha, \beta <\omega_1,\, r\in A_{\alpha,\beta},\ r =\langle  f_r, \langle  T_r, b_r\rangle\rangle\right\}
    \end{equation*}

    where $\pi_1\colon\kappa\times\omega_1\rightarrow\kappa$ is the projection on the first coordinate. The ordinal $\delta$ is the supremum of a bounded set of ordinals below $\kappa$ and thus $\delta < \kappa$. In particular,  $\dot{T}$ is equivalent to a $\mathbb{CK}_\delta$-name.

    Since $\mathbb{CK}_\delta$ adds the new tree $\dot{T}$, we can ask how many branches does $\dot{T}^G$ have in $M[G]$, for $G\subset \mathbb{CK}_\delta$ a generic filter. An upper bound for the number of branches is $\left(2^{\left| \mathbb{CK}_\delta\right|}\right)^{\aleph_1}$, this is an upper bound because the number of anti-chains is at most $2^{\left| \mathbb{CK}_\delta\right|}$ hence the number of nice names of branches is $\left(2^{\left| \mathbb{CK}_\delta\right|}\right)^{\aleph_1}$, and since $\kappa$ is inaccessible $\left(2^{\left| \mathbb{CK}_\delta\right|}\right)^{\aleph_1} < \kappa$. $\mathbb{C}_{\omega_1, <\kappa}$ collapses all cardinals between $\omega_1$ and $\kappa$ (see \cite[Chapter 15]{Jech}), thus if no branches are added by the rest of the forcing the tree $\dot{T}$ would have $\aleph_1$ branches in the generic extension (by $\mathbb{CK}_\kappa$). Since it is assumed that
$\Vdash_{\mathbb{CK}_\kappa } \dot{T} \text{ is Kurepa},$

    if we manage to prove

    \begin{equation*}
        \mathbb{CK}_\kappa \simeq \mathbb{CK}_\delta \ast \dot{\mathbb{Q}}
    \end{equation*}

    for some name of a forcing $\dot{\mathbb{Q}}$, it would imply that $\dot{\mathbb{Q}}$ adds $\kappa$ branches to $\dot{T}$. In order to prove the above we follow the proof of lemma \ref{lem:K_sim_K_delta}, first we build a projection map $\pi\colon\mathbb{CK}_\kappa\rightarrow \mathbb{CK}_\delta$. Define for $f\in \mathbb{C}_{\omega_1, <\kappa}$ the restriction $f\restriction \delta = f\cap (\delta \times \omega_1\times \delta)$ and let

    \begin{equation*}
        \pi(p) = \langle  f_p\restriction \delta, T_p, b_p\restriction \delta \rangle , \quad p = \langle  f_p, T_p, b_p\rangle \in \mathbb{CK}_\kappa
    \end{equation*}

    We prove that $\pi$ is a projection. Let $p, q\in \mathbb{CK}_\kappa$ and assume $p\leq q$, we use the argument of lemma \ref{lem:K_sim_K_delta} and conclude that $\pi(p)\leq\pi(q)$. Let $p\in \mathbb{CK}_\kappa$, and let $q\leq \pi(p)$ and consider $ p_\star = \langle  f_q\cup f_p, T_q, b_q\cup b_p^\star \rangle$, where $b_p^\star :\kappa \rightarrow T_{q,\max}$, $\dom b_p\cap [\delta, \kappa) = \dom \, b_p^\star$ and $b_p^\star(\alpha) \geq b_p(\alpha)$. $ p_\star \leq p$ is a condition in $\mathbb{CK}_\kappa$ and $\pi(p_\star)\leq q$. Hence, $\pi$ is a projection. Let $H\subset\mathbb{CK}_\delta$ be some generic filter. Let us define $\mathbb{Q}$

    \begin{equation*}
        \mathbb{Q} = \left\{ p\in \mathbb{CK}_\kappa\, \Big|\, \pi(p) \in H \right\}.
    \end{equation*}

    Denote by $\dot{\mathbb{Q}}$ a $\mathbb{CK}_\delta$-name for $\mathbb{Q}$. As in lemma \ref{lem:K_sim_K_delta}, we have a dense embedding $\iota: \mathbb{CK}_\kappa\rightarrow\mathbb{CK}_\delta\ast \dot{\mathbb{Q}}$, defined via $\iota(p) = (\pi(p), \check{p})$. We define another embedding $\iota':\iota(\mathbb{CK}_\kappa)\rightarrow\mathbb{CK}_\delta\ast \dot{\mathbb{CK}}_{[\delta, \kappa)}$, via $\iota'(\pi(p), \check{p}) = (\pi(p), \check{f_p}\restriction [\delta, \kappa), \check{b_p}\restriction [\delta, \kappa))$. This is a dense embedding (same exact proof as lemma \ref{lem:K_sim_K_delta}). Hence

    \begin{equation*}
        \mathbb{CK}_\kappa\simeq\mathbb{CK}_\delta\ast\dot{\mathbb{CK}}_{[\delta, \kappa)}
    \end{equation*}

    Let $\dot{b}$ be a name of a new branch forced by $\dot{\mathbb{CK}}_{[\delta, \kappa)}$ to be in the model. We repeat the argument and get that $ \exists \zeta,\,  \delta < \zeta <\kappa$ such that $\dot{b}$ is equivalent to a $\mathbb{CK}_{[\delta, \zeta)}$-name. Denote by $\mathbb{R}_1 =\mathbb{CK}_\delta$, $\mathbb{R}_2 = \mathbb{CK}_{[\delta, \zeta)}$ and $\mathbb{R}_3 = \mathbb{CK}_{[\zeta, \kappa)}$ then

    \begin{equation*}
        \mathbb{CK}_\kappa \simeq \mathbb{R}_1\ast \dot{\mathbb{R}}_2 \ast \dot{\mathbb{R}}_3
    \end{equation*}

    Consider the forcing, $\mathbb{R}_1\ast \dot{\mathbb{R}}_2 \ast \dot{\mathbb{R}}_3\ast \check{\dot{\mathbb{R}}}_2$. Where $\check{\dot{\mathbb{R}}}$ is the $\mathbb{R}_1$-name for $\mathbb{R}_2$ i.e.

    \begin{equation*}
        (*)\, \Vdash_{\mathbb{R}_{1}\ast\dot{\mathbb{R}}_2\ast\dot{\mathbb{R}}_3} \forall x\in \check{\dot{\mathbb{R}}}_2, \exists \dot{\tau} = x, \exists \check{\sigma}\in \dot{\mathbb{R}}, \dot{\tau} = \check{\sigma}
    \end{equation*}

    From now on the $\check{}$ would be omitted from the second $\dot{\mathbb{R}}_2$. Since iterated forcing is associative we have that

    \begin{equation*}
        \mathbb{R}_1\ast \dot{\mathbb{R}}_2 \ast \dot{\mathbb{R}}_3\ast \dot{\mathbb{R}}_2 \simeq \mathbb{R}_1\ast \left(\dot{\mathbb{R}}_2 \ast \dot{\mathbb{R}}_3\ast \dot{\mathbb{R}}_2\right)
    \end{equation*}
\end{proof}

\begin{lemma}
    \begin{equation*}
        \mathbb{P}_1 =\mathbb{R}_1\ast \left(\dot{\mathbb{R}}_2 \ast \dot{\mathbb{R}}_3\ast \dot{\mathbb{R}}_2\right) \simeq \mathbb{R}_1\ast \left(\left(\dot{\mathbb{R}}_2 \ast \dot{\mathbb{R}}_3\right)\times \dot{\mathbb{R}}_2\right) =\mathbb{P}_2
    \end{equation*}
\end{lemma}

\begin{proof}
    Denote, by $\dot{\mathbb{R}}^\star \equiv \dot{\mathbb{R}}_2\ast \dot{\mathbb{R}}_3$. Let $(p_1, \dot{r}_1)\in \mathbb{P}_2)$, then $\exists p_2\leq p_1$ s.t.

    \begin{equation*}
        p_2\Vdash \dot{r_1} = (\dot{r}, \dot{r}_2)\in \dot{\mathbb{R}}^\star\times \dot{\mathbb{R}}_2
    \end{equation*}

    From $(*)$, $\exists p_3\leq p_2$ s.t.

    \begin{equation*}
        p_3 \Vdash \dot{r} = (\dot{q}_1, \dot{q}_2)\in \dot{\mathbb{R}}_2 \ast \dot{\mathbb{R}}_3 \wedge (\dot{q_1}, \dot{q_2}) \Vdash \dot{r}_2 = \dot{\check{\sigma}} \in \dot{\mathbb{R}}_2
    \end{equation*}

    Thus, the condition $(p_3, ((\dot{q}_1, \dot{q}_2), \dot{\check{\sigma}}))\leq (p_1, \dot{r}_1)$ and $\mathbb{R}_1\ast \dot{\mathbb{R}}_2 \ast \dot{\mathbb{R}}_3\ast \dot{\mathbb{R}}_2$ is densely embedded inside $\mathbb{R}_1\ast \left(\left(\dot{\mathbb{R}}_2 \ast \dot{\mathbb{R}}_3\right)\times \dot{\mathbb{R}}_2\right) =\mathbb{P}_2$ hence equivalent as a forcing notion.
\end{proof}
\begin{claim}\label{claim:R_2^2=R_2}
$\mathbb{R}_2 \times \mathbb{R}_2 \cong \mathbb{R}_2$. 
\end{claim}
\begin{proof}
Recall that $\mathbb{R}_2 = \mathbb{C}_{\omega_1, [\delta, \zeta)} \ast \mathbb{K}_{ [\delta, \zeta)}$. So, it is enough to show that $\mathbb{C}_{\omega_1, [\delta, \zeta)} \times \mathbb{C}_{\omega_1, [\delta, \zeta)} \cong \mathbb{C}_{\omega_1, [\delta, \zeta)}$ and that in the generic extension, $\mathbb{K}_{ [\delta, \zeta)} \times \mathbb{K}_{ [\delta, \zeta)} \cong \mathbb{K}_{ [\delta, \zeta)}$. 

Both isomorphisms hold by reordering the coordinates in the collapse part and in the branches.
\end{proof}
\begin{proof}[Back to the proof of Theorem \ref{thm:every_Kurepa_not_sealed}]
    Let $G\subset \mathbb{R}_1$ be a generic filter, then in $V[G]$ the realization $\left(\left(\dot{\mathbb{R}}_2^G \ast \dot{\mathbb{R}}_3^G\right)\times \dot{\mathbb{R}}_2^G\right)$ is a product. Let $H_1\subset\left(\dot{\mathbb{R}}_2^G \ast \dot{\mathbb{R}}_3^G\right)$ and $H_2\subset \dot{\mathbb{R}}_2^G$ be generic filters, then $H_1, H_2$ are mutually generic thus $\dot{b}^{G,H_1}\neq \dot{b}^{G,H_1,H_2}$ and since $\dot{b}$ is a $\mathbb{R}_2$-name of a new branch and $V[G][H_1]$ is the model after forcing with $\mathbb{CK}_\kappa$ we have that $\dot{\mathbb{R}}_2^{G}$ adds a new branch to $T$. The cardinality of $\dot{\mathbb{R}}_2^G$ is at most $\zeta^{(\zeta\cdot\aleph_1)}\cdot\zeta^{\aleph_2}$ and since $\zeta <\kappa$ and $\kappa$ is strongly inaccessible $\zeta^{\zeta\cdot\aleph_1}\cdot\zeta^{\aleph_2} < \kappa$ therefore in $V[G][H_1]$ the cardinality is $\left|\dot{\mathbb{R}}_2^G\right| = \aleph_1$.
\end{proof}  

Next, we would like to show a variant of Theorem \ref{thm:every_Kurepa_not_sealed}, in which $\diamondsuit^+$ holds, but no Kurepa tree is sealed. 

\begin{theorem}
Assume that there is exists an inaccessible cardinal. Then, there is a generic extension in which $\diamondsuit^+$ holds and for every Kurepa tree $T$ there is a $\sigma$-distributive forcing notion of cardinality $\aleph_1$ that adds a branch to $T$.
\end{theorem}
\begin{proof}
The construction is extremely similar to the one in the proof of Theorem \ref{thm:every_Kurepa_not_sealed}, so let us focus of the differences. 

Assume $\mathrm{GCH}$. Let $\Phi \colon \eta \to H(\omega_2)$.\footnote{We eventually would like to take $\Phi$ to be a bookkeeping function, i.e.\ $\eta=\omega_2$ and $\Phi$ is a surjection in which every value appears unboundedly. We need to treat a slightly more general case first.}  

Let $\mathbb{D}_{\Phi}$ be the standard forcing for adding $\diamondsuit^+$, assuming that $\Phi$ is surjective. Namely, $\mathbb{D}$ is a countable support iteration of length $\eta$, $\langle \dot{\mathbb{B}}_\alpha, \mathbb{D}_\beta \mid \alpha < \beta \leq \eta\rangle$ such that $q \in \mathbb{B}_0$ if and only if $q$ is a function with domain $\rho < \omega_1$ and $q(\zeta) \subseteq \mathcal{P}(\zeta)$ and countable for every $\zeta < \rho$. We order the forcing by end extension. Let $\langle A_\alpha \mid \alpha < \omega_1\rangle$ be the union of conditions in the generic filter for $\mathbb{B}_0$. 

Let us define $\dot{\mathbb{B}}_\alpha$ to be the trivial forcing unless $x = \Phi(\alpha)$ is a $\mathbb{D}_\gamma$-name for a subset of $\omega_1$, and in that case $\dot{\mathbb B}_\alpha$ consists of closed bounded sets $a$ such that $\forall \rho \in a,\, a \cap \rho,\, x \cap \rho \in A_\rho$.  

In \cite[Section 12]{CummingsForemanMagidor}, this forcing is shown to have a dense $\sigma$-closed subset, and it forces $\diamondsuit^+$. In particular, by a theorem of Jensen, there is a Kurepa tree in the generic extension.

Let us define $\mathbb{CD}_\kappa$ to be $\mathbb{C}_{\omega_1, <\kappa} \ast \mathbb{D}_{\Phi}$ for some bookkeeping function $\Phi$. 
Let us remark that the main properties of the forcing, except for forcing the strong diamond, are true for \emph{any} function $\Phi$. 
\begin{remark}
Assume $\CH$. For every ordinal $\eta$ and function $\Phi \colon \eta \to H(\omega_2)$, $\mathbb{D}_{\Phi}$ is $\omega_2$-c.c.\ and contains a $\sigma$-closed dense subset.
\end{remark}
So if $G$ is generic for $\mathbb{CD}_\kappa$, $\aleph_1^{V[G]} = \aleph_1^V$ and $\aleph_2^{V[G]} = \kappa$.

Let $T$ be a Kurepa tree in $V[G]$. By the same chain condition arguments, there is an initial segment of the collapse and an initial segment of the iteration of $\mathbb{D}$. Note that in order for this to be true, we must be careful, as the bookkeeping function $\Phi$ can be evaluated only in the full generic extension by $\mathbb{C}_{\omega_1,<\kappa}$. 

We argue that by the chain condition of the Levy collapse, for every initial segment of the iteration $\mathbb{D}$ there is an initial segment of the Levy collapse realizing it. 

So, by the previous argument, there is a forcing notion $\dot{\mathbb{Q}}$ which is a tail of both the Levy collapse and the diamond iteration adding a branch to $T$. The next steps of the proof are slightly more involved, as the final segment of the forcing $\mathbb{D}$ does not decompose into a product as in the $\mathbb{K}_\kappa$ case, but rather to an iteration. In particular we need to replace Claim \ref{claim:R_2^2=R_2}, by something slightly different.
 
Let us decompose our forcing $\mathbb{CD}_{\kappa} = \mathbb{R}_1 \ast \mathbb{R}_2 \ast \mathbb{R}_3$ as in the previous proof. So, $\mathbb{R}_1$ is the component that adds the tree $T$, $\mathbb{R}_2$ is adding a new branch and $\mathbb{R}_3$ is the rest of the forcing. Without loss of generality, $\mathbb{R}_1$ is of the form $\mathbb{C}_{\omega_1, <\delta} \ast \mathbb{D}_\delta$ for $\delta > \omega_1$. $\mathbb{R}_2$ is of the form $\mathbb{C}_{\omega_1, [\delta,\delta')} \ast \mathbb{D}_{[\delta, \delta')}$ and $\mathbb{R}_3 = \mathbb{C}_{\omega_1, [\delta',\kappa)} \ast \mathbb{D}_{[\delta', \kappa)}$.
Clearly, in the generic extension, $|\mathbb{R}_2| = \aleph_1$. 
 
The following claim serves as a replacement for Claim \ref{claim:R_2^2=R_2}.
\begin{claim}
There is a name for a function $\Phi' \colon \omega_2 + \delta' - \delta \to H(\omega_2)$ such that 
\[\mathbb{C}_{\omega_1, <\kappa} \ast \mathbb{D}_{\Phi'} \cong \mathbb{C}_{\omega_1,<\kappa} \ast \mathbb{D}_{\Phi} \ast \mathbb{R}_2\]
\end{claim}
\begin{proof}
Let us decompose $\mathbb{R}_2$ into $\mathbb{C}_{\omega_1, [\delta,\delta')} \ast \mathbb{D}_{[\delta, \delta')}$. Recall that 
\[\mathbb{C}_{\omega_1, <\kappa} \times \mathbb{C}_{\omega_1, [\delta,\delta')} \cong \mathbb{C}_{\omega_1, <\kappa}.\]

Let $G_1 \times G_2 \subseteq \mathbb{C}_{\omega_1, <\kappa} \times \mathbb{C}_{\omega_1, [\delta,\delta')}$ be $V$-generic (which is equivalent to a $\mathbb{C}_{\omega_1,<\kappa)}$-generic). 

Let us define $\Phi'$ in $V[G_1][G_2]$ as follows. $\Phi'(\alpha)^{G_1 \times G_2} = \Phi(\alpha)^{G_1}$ for $\alpha < \omega_2$. For $\alpha \in [\omega_2, \omega_2 + (\delta' - \delta))$, we define $\Phi'(\alpha) = \Phi(\alpha - \omega_2)^{G_1 \restriction \delta \times G_2}$. 

So, $\Phi'$ is a function from $\kappa$ to $H(\omega_2)^{V[G_1][G_2]}$, 
and the iteration that it induces is equivalent to 
$\mathbb{D}_{\Phi}^{G_1} \ast \mathbb{D}_{[\delta, \delta')}^{G_1 \restriction \delta \times G_2}$.
\end{proof}

Since $\mathbb{D}_{\Phi'}$ contains a $\sigma$-closed dense subset, the whole forcing does not add countable sequences of ordinals and therefore $\mathbb{R}_2$ is forced to be $\sigma$-distributive, as wanted.
\end{proof}
\section{The Sealing Poset}\label{section:sealing-trees-forcing}
In this section we prove the first part of the main theorem of this work. Here we present the forcing notion $\mathbb{S}$ (called the sealing poset) and prove that $\mathbb{S}$ is proper for $\mathcal{S}\cup \mathcal{T}$. From this property we conclude that it preserves $\omega_1$ and $\omega_2$. 

The proof for properness relies heavily on reflection between a model and its elementary submodel. Before we delve into the details of our case, let us demonstrate this method on a simpler case. We start by providing a proof of a theorem by Baumgartner. The proof itself is almost equivalent to Baumgartner's original proof (using ultrafilters), but we are replacing the role of the ultrafilter by an elementary substructure and emphasize the hidden compactness argument.\footnote{We would like to thank Martin Goldstern for pointing our attention to this connection.} 
\begin{definition}
    Let $T$ be an Aronszajn tree, we define the specializing forcing notion
    \begin{equation*}
        \mathbb{A} = \left\{ f:T\rightarrow \omega\ \Big|\ |f|<\aleph_0 \wedge \forall x,y\in \dom f, x \parallel y\rightarrow x = y\vee f(x)\neq f(y) \right\}
    \end{equation*}
\end{definition}

\begin{theorem}[Baumgartner]
    Let $T$ be an Aronszajn tree then $\mathbb{A}$ adds a specializing function of $T$ without collapsing $\aleph_1$.
\end{theorem}
\begin{proof}[Proof. (Folklore)]
    In his proof, Baumgartner showed that $\mathbb{A}$ is c.c.c.\ which implies that $\omega_1$ is preserved. To achieve the same we prove that $\mathbb{A}$ is proper for every countable elementary submodel (which also implies $\omega_1$ is preserved). The proof itself can be adjusted to obtain the stronger result.

    The cardinality of $\mathbb{A}$ is $\aleph_1$. 
    Let $M\prec H((2^{\aleph_1})^ +)$ be a countable elementary submodel. 
    Assume $\mathbb{A}$ is not proper for $M$ then $\exists  p\in \mathbb{A}$ such that 
    $\forall q\leq p$, $q$ is not $M$-generic. 
    Thus, $p$ itself is not $M$-generic then $\exists D\in M, D\subset\mathbb{A}$ dense and 
    $\exists p' \leq p$ such that $\forall r\in D\cap M,\, r\perp p'$. 
    Since $D$ is dense $\exists d\in D$ such that $d\leq p'$ thus $\forall r\in D\cap M$ we have $r\perp d$. 
    
    Consider the set
    \begin{equation*}
        \mathfrak{D} = \left\{ r\in D\, \Big|\, |d| = |r| \wedge r\leq d^M\right\},
    \end{equation*}
    where $d^M = d\cap M$. Since $d^M\in M$ (as $d^M$ is finite), $\mathfrak{D}$ is definable in $M$ therefore $\mathfrak{D}\in M$. 
    
	Note that $d\notin M$, and thus the set $\mathfrak{D}$ is \emph{large} from the prespective of $M$. The following definition is the exact formulation of this fact.
\begin{definition}
    Let $X\subset T$, we say $X$ is \underline{unbounded} if 
    \[\sup\left\{ \level(x)\ \Big|\ x\in X \right\} = \omega_1.\]
\end{definition}

\begin{lemma}
    \label{lem:X_is_unbounded}
    Let $X\in M,\ X\subset T$ and assume $\exists x\in X\backslash M$. Then $X$ is unbounded.
\end{lemma}

\begin{proof}
    Let $X$ be such a set, we choose some element in $X\backslash M$ and denote it by $y$. Assume that $X$ is bounded in $T$. Let $\alpha = \sup \left\{ \level(x)\ \Big|\ x\in X \right\}$, since $X$ is bounded $\alpha < \omega_1$. Since $M\prec H(\omega_3)$ and $X\in M$ then $M$ also thinks that $X$ is bounded. Let $\sup \left\{ \level(x)\ \Big|\ x\in X \right\} =\beta <\omega_1^M$.  Since $y\notin M$ and $y\in X$ then $\exists x\in X$ s.t.\ $\level(x) \geq\omega_1^M > \beta$. Thus

    \begin{equation*}
        H\left( \omega_2 \right) \models \exists x\in X, \level(x) > \beta
    \end{equation*}

    And from elementarity
    \begin{equation*}
        M \models \exists x\in X, \level(x) > \beta
    \end{equation*}
    which is a contradiction.
\end{proof}

    Let us enumerate the domains of the conditions in $\mathfrak{D}$ in a way 
    that the enumeration of $d^M$ is constant across all conditions in $\mathfrak{D}$. 
    Then, the reason for the incompatibility of a condition $r\in D\cap M$ and $d$ must be 
    $\exists i,j <|d|$ such that $r_j\leq d_i$ and $r(r_j) = d(d_i)$.

    We show that there exists a \emph{coloring} $P_{ij}$ of $\mathfrak{D}$ for $i,j < |d|$ with the following property

    \begin{equation*}
        (*)\quad r^1, r^2\in P_{ij} \rightarrow r^1_j \parallel r^2_j \wedge r^1(r^1_j) = r^2(r^2_j) = d(d_i).
    \end{equation*}
    It will be convenient for us to allow an element to be colored with more than one color.  
    Let $X\subset\mathfrak{D}\cap M$ be a \textbf{finite} set and consider the set

    \begin{equation*}
        P_{ij}^X = \left\{ r\in X\, \Big|\, r_j\leq d_i \wedge r(r_j) = d(d_i)\right\}.
    \end{equation*}

    Since every element in $X$ is incompatible with the condition $d$, it has to be in one of the $P_{ij}^X$. In other words, $X = \bigcup_{ij < |d|} P_{ij}^X$.

    While the sets $P_{ij}^X$ are uniformly definable in $H(\omega_3)$, they are not uniformly definable in $M$. Nevertheless, they are finite and thus belong to $M$. 


	In particular,
    \begin{equation*}
        M \models \exists P_{ij}^X,\, \left(X = \bigcup_{ij < |d|} P_{ij}^X \right) \wedge \left(P_{ij}^X\text{ has property } (*)\right)
    \end{equation*}
	As this is true for every finite $X \subseteq \mathcal{D}$ in $M$,
    \begin{equation*}
        M \models \forall X \subseteq \mathcal{D},\,\text{finite } \exists P_{ij}^X,\, \left(X = \bigcup_{ij < |d|} P_{ij}^X \right) \wedge \left(P_{ij}^X\text{ has property } (*)\right)
    \end{equation*}
	Let us now apply the compactness theorem for propositional logic. Let adds an atomic proposition $c_{r, i, j}$ for each $r \in \mathcal{D}$ and $i, j < |d|$. We look at the collection of propositions encoding $(*)$, which are $\mathcal{T} = \mathcal{T}_0 \cup \mathcal{T}_1$, 
	\[\begin{matrix}
	\mathcal{T}_0 & = & \{\neg (c_{r,i,j} \wedge c_{r', i, j}) \mid & r, r'\in \mathcal{D},\, i, j < |d|, \\ 
	& & &  r(j) \perp r'(j) \vee \neg (r(j) = r'(j) = d(i))\} \\
	\mathcal{T}_1 & = & \{\bigvee_{i,j<|d|} c_{r, i, j} \mid r \in \mathcal{D}\} 
	\end{matrix}\] 




	By the remarks above
    \begin{equation*}
        M\models \mathfrak{T} \text{ is finitely satisfiable.}
    \end{equation*}



	Therefore,
    \begin{equation*}
        M\models P_{ij} \text{ is a coloring of } \mathfrak{D}\text{ with property } (*).
    \end{equation*}

    Hence, by elementarity

    \begin{equation*}
        H(\omega_3)\models P_{ij} \text{ is a coloring of } \mathfrak{D}\text{ with property } (*)
    \end{equation*}

    Hence, $\exists i,j < |d|$ such that $d\in P_{ij}$. We show that $d\in P_{ij}$ either leads to a contradiction or to the existence of a subset $X$ of $P_{ij}$ such that $d\in X$ and all elements of $X$ there $j$'th element is compatible with the $i$'th of $d$. First consider the case when $d_j\notin M$, in this case, we have that  $P_{ij}$ is unbounded and then since $\forall r^1,r^2\in P_{ij},\, r^1_j\parallel r^2_j$, $P_{ij}$ defines the branch

    \begin{equation*}
        B = \bigcup\left\{ seg(r_j)\, \Big|\, r\in P_{ij} \right\}
    \end{equation*}

    This is a branch of $T$ in $M$, but from elementary $M$ also thinks that $T$ is Aronszajn, contradiction. We remain with the case $d_j\in M$, consider the set

    \begin{equation*}
        X = \left\{ r\in P_{ij}\, \Big|\, r_j = d_j \wedge r(r_j) = d(d_j)\right\}
    \end{equation*}

    Let $r\in X$ be some condition, then since $r_j = d_j\wedge r(r_j) = d(d_j)$ we have that $r_j$ can't interfere with $d_i$. In particular, for every $r \in P_{i,j}$, the reason for its incompatibility with $d$ is never that the $j$-th member of $r$ lie below the $i$-th member of $d$ (in fact, they are equal). 
    
    Let us repeat the argument, without the color $(i,j)$, on the set $P_{i,j}$. We again obtain a coloring with $|d|^2 - 1$ many colors, and can either get a cofinal branch, or a small unbounded set of conditions for which certain $(i',j')$ can never witness incompatibility.  
    
	After a finitely many applications of this argument we would get a set (in $M$) with $d$ in it such that all the conditions in it from $M$ are compatible with $d$. Since $M$ would think that this set is not empty we get a condition in $M$ compatible to $d$, contradicting our hypothesis.
\end{proof}

We now begin the main proof of this section.

\begin{definition}
    \label{def:forcing_poset}
    Let $T$ be a Kurepa tree, define the \underline{sealing poset of $T$}, $\mathbb{S}$: \hfill

    $p\in\mathbb{S}$ iff $p =\langle  f_p, \vec{M}_p\rangle$ where $f_p : T\cup \left[T\right]\rightarrow\omega$ and $\vec{M}_p \in \mathbb{M}$ is a finite sequence of elementary submodels of two types. We also require that our conditions have the following properties:
    \begin{enumerate}
        \item {\bf Weak Monotonicity.} $\forall x\leq_T y\in \dom f_p,\ f_p\left( x \right)\leq f_p \left( y \right)\wedge \forall x,b\in \dom f_p,b\in \left[T\right]\wedge x\in b\rightarrow f_p\left( x \right)\leq f_p\left( b \right)$.
        \item \label{condition:2} {\bf Closure under branches.} $\forall N\in \vec{M}_p\forall b\in N\left( \exists t\in \left(\dom f_p\backslash N\right) \cap b \right) \rightarrow \left( b\in\dom f_p\wedge f_p\left( b \right) = f_p\left( t \right) \right)$.
        \item {\bf Closure under meets.} $\forall x_1, x_2 \in \dom f_p,\ x_1 \wedge x_2\in \dom f_p$.
        \item {\bf Inclusion of the root.} $\emptyset_{root}\in \dom f_p\wedge f_p \left( \emptyset_{root} \right) = 0$.
    \end{enumerate}

    $\mathbb{S}$ is ordered by reverse inclusion.
\end{definition}

\begin{remark}
    We want our forcing to add a generic function $f:T\cup[T]\rightarrow\omega$ without collapsing $\omega_1$ or $\omega_2$. Now in $M[G]$ if a new generic branch $b$ is added via some other forcing we want $f\restriction b$  to be the cause of $\omega_1$ collapse.

    Since we require monotonicity of $f$ we expect that $f\restriction b$ would then be unbounded.

    Condition \ref{condition:2} is needed to ensure properness. We need $N$ to know the bound on the value of a branch $b\in N$ that has a bound (from monotonicity) outside of $N$.

    Condition 3 is needed in case there are two branches which are determined by condition \ref{condition:2} have a common ancestor which determines them both. In this case we want to make sure both branches are compatible.

    Condition 4 sets the root to 0, which removes a redundant degree of freedom for our function.
\end{remark}

The following lemma is the technical heart of the main theorem of this paper.
\begin{lemma}
    \label{lemma:P_is_proper}
    $\mathbb{S}$ is proper for $\mathcal{S}$. In particular $\mathbb{S}$ does not collapse $\omega_1$.
\end{lemma}

\begin{proof}[Proof of Lemma \ref{lemma:P_is_proper}]
    To prove $\mathbb{S}$ is proper we need to show that for club many $M\in \mathcal{S}$, $\forall p\in \mathbb{S}\cap M$ there exists an extension $q\leq p$ which is an $M$-generic condition.
    Let $\lambda = \left(2^{|\mathbb{S}|}\right)^{+}$. Let us pick $M\in \mathcal{S}$ such that $\exists M'\prec H(\lambda)\wedge M'\cap H(\omega_2) = M$. Let us assume, towards a contradiction, that $\exists p \in M$ such that $\forall q \leq p$, $q$ is not $M$-generic. 
    
    Since $p\in M$, the sequence of models is also in $M$. From corollary \ref{cor:p_cup_M} we have a sequence of models $\vec{M}_r\leq \vec{M}_p$ which is the closure of $\vec{M}_p\cup \left\{ M \right\}$ under intersections. Consider $\tilde{r} = \langle  f_p,\vec{M}_r\rangle$. LEt us prove that $\tilde{r}$ is a condition stronger than $p$. Thus, it has no extension to an $M$-generic condition. 

\begin{claim}
    $\tilde{r}$ is a condition and $\tilde{r} \leq p$.
\end{claim}

\begin{proof}
    In order to show that $\tilde{r}\in \mathbb{S}$ all we need to show is that condition (\ref{condition:2}) from definition \ref{def:forcing_poset} holds, since we only modified the model sequence. Let $N\in \vec{M}_{\tilde{r}}$, by Corollary \ref{cor:p_cup_M}, $N$ is an intersection of models $N = \bigcap_{i < m}N_i$ where $N_i\in \vec{M}_p\cup\left\{ M \right\}$. Let $b\in N$ and $x\in b\cap\dom f_p\backslash N$ then $\exists i < m$ such that $x\notin N_i$ and $b \in N_i$, thus from condition (\ref{condition:2}) $b$ would have already been in $\dom f_p$ satisfying $f_p\left( b \right) = f_p\left( x \right)$ ($N_i \neq M$ since $\dom f_p \subset M$).
\end{proof}

    In particular $\tilde{r}$ is not $M$-generic. Thus, there is $q \leq \tilde{r}$ and $D\in M$ dense such that $\forall r \in M\cap D, r\perp q$. Since $D$ is dense $\exists d\in D$, $d\leq \tilde{r}$. This condition $d$ also satisfies $\forall r\in M\cap D,\ r\perp d$, and in particular $d\notin M$. 

\begin{lemma}\label{lemma:closure-under-branches-meets-and-projections}
Let $E$ be the set of all conditions $r = \langle f_r, \vec{N}_r\rangle \in \mathbb{S}$ such that for every $x \in \dom f_r$ and $N \in \vec{M}_r$ if $\Lev(x) > (N \cap \omega_1)$ then the unique element in the tree $T$ in level $N \cap \omega_1$ below $x$, $\bar{x}$, belongs to $\dom f_r$. 

Then $E$ is dense. 
\end{lemma}
\begin{proof}
First, let us fix terminology. For an element $x \in T \cup [T]$, and an ordinal $\alpha < \Lev(x)$ (which is $\omega_1$ if $x\in [T]$), the \emph{projection} of $x$ to level $\alpha$ is the unique $\bar{x}$ such that $\Lev(\bar{x}) = \alpha$ and $\bar{x} \leq_T x$ (or $\bar{x} \in x$, where $x\in [T]$). We denote this $\bar{x}$ by $\proj_\alpha(x)$.

Let $p$ be an arbitrary condition. We must show that we can find an extension of $f_p$ to a weakly monotonic function $f_r$ such that its domain is closed under meets, projections to levels of the form $N \cap\omega_1$ for $N\in \vec{M}_r$, and closed under adding branches, as in the definition of the conditions of $\mathbb S$. 
We will do that in a finite sequence of extensions.



Let $B = \dom f_p$.
We define recursively a sequence of elements of $[T] \cup T$, $t_0, t_1, \dots$, as follows.  

Given $\{t_0, \dots, t_{n-1}\} \subseteq T \cup [T]$, let us look at the collection $C_n$ of all $t\in [T] \cup T \setminus B \cup \{t_0,\dots, t_{n-1}\}$ such that either:
\begin{itemize}
\item there are $s_0, s_1 \in B \cup \{t_0,\dots, t_{n-1}\}$ such that $t = s_0 \wedge s_1$, or
\item there is $N \in \vec{M}_p$ and $w \in B \cup \{t_0,\dots, t_{n-1}\}$ such that $t = \proj_{N \cap \omega_1}(w)$, or
\item there is $N \in \vec{M}_p$ and $t \in [T] \cap N$ and $w \in B \cup \{t_0,\dots, t_{n-1}\}$ , such that $w \in t \setminus N$. 
\end{itemize}
We define an ordinal $\alpha$ to be a \emph{level of occurrence} of an element $t\in C_n$ if $\alpha = \Lev(t)$ in the first two cases and $\alpha = N \cap \omega_1$ in the last one. We pre-order the elements of $C_n$ first by their minimal level of occurrence, then by the reason that they appeared (closure under meets before projections, and projections before adding cofinal branches). Note that an element in $C_n$ might have several reasons for being in $C_n$, and in the case that it is a branch, perhaps even many levels of occurrence. In those cases, we associate to this element the minimal level of occurrence and the minimal reason for appearance. 

Finally, we take $t_n \in C_n$ to be an element with the minimal level of occurrence and among those elements, a minimal reason of appearance. Let $\gamma_n$ be the level of occurrence of $t_n$.

We want to show that this process terminates after finitely many steps. 
\begin{claim}
Let $n_*$ be the least natural number such that $C_n = \emptyset$, if there is one, and $\omega$ otherwise.
\begin{enumerate}
\item The sequence $\langle \gamma_n \mid n < n_*\rangle$ is weakly increasing. 
\item Let $\gamma$ be an ordinal such that $\exists n\, \gamma = \gamma_n$ and let $k_0$ be the least such number. Then, $|\{n < n_* \mid \gamma_n = \gamma\}|$ is bounded by $\binom{|B| + k_0}{2} + (|B| + k_0) \cdot (1 + |\{N \in \vec{M}_p \mid N \cap \omega_1 = \gamma\}|)$ and in particular, finite. 
\item Let us assume that $\gamma_n = N \cap \omega_1$, $N\in \vec{M}_p$ and let $\delta > \gamma$ be the least ordinal of the form $N' \cap \omega_1$, for $N' \in \vec{M}_p$. Let us assume, moreover, that $n$ is maximal such that $\gamma_n = N \cap \omega_1$. Then, for $\ell = \binom{|B| + n}{2}$, if $n + \ell + 1 < n_*$, then $\gamma_{n + \ell + 1} \geq \delta$.  
\end{enumerate}
Therefore, $n_* < \omega$.
\end{claim}
\begin{proof}
\begin{enumerate}
\item Let us show first that $\langle \gamma_n \mid n < \omega\rangle$ is weakly increasing. Let us look at $t_n$ and let us assume, towards a contradiction, that $\gamma_n <\gamma_{n-1}$. Let us go over all possible reasons for $t_n$ to appear:
\begin{itemize}
\item There are $s_0, s_1 \in B \cup \{t_0,\dots, t_{n-1}\}$ such that $t_n = s_0 \wedge s_1$. Clearly, if $s_0\neq t_{n-1}$ and $s_1 \neq t_{n-1}$, we would have picked them earlier, so without loss of generality, $s_1 = t_{n-1}$. Let us check where $t_{n-1}$ was originated: 
\begin{itemize} 
\item $t_{n-1} = s_2\wedge s_3$. Then $t_n = s_0 \wedge s_2 \wedge s_3$. By Lemma \ref{lem:wedge_properties}, \ref{wedge_property_idenpotency}, $t_n \in \{s_0 \wedge s_2, s_0 \wedge s_3\}$ and thus should have appeared before $t_{n-1}$.
\item $t_{n-1} = \proj_{\gamma_{n-1}}(w)$. Then $t_n = s_0 \wedge w$ (using the assumption that $\gamma_n < \gamma_{n-1}$) and again should have appeared before $t_{n-1}$.
\item $t_{n-1}$ is a branch from $N \cap [T]$ that goes through $w \in t_{n-1} \setminus N$. In this case, $\gamma_{n-1} = N \cap \omega_1 \leq \Lev(w)$. Since $\Lev(t_n) < \gamma_{n-1}$, we obtain that $t_n = s_0 \wedge w$.  
\end{itemize}
\item There is $w\in B \cup \{t_0,\dots, t_{n-1}\}$ such that $t_{n} = \proj_{N\cap \omega_1}(w)$. Again, if $w \neq t_{n-1}$, there is nothing to show. So, $w = t_{n-1}$ and we again split into cases based on where $t_{n-1}$ came from:
\begin{itemize}
\item $t_{n-1} = s_0 \wedge s_1$. In this case, $t_n = \proj_{N \cap \omega_1}(s_0) = \proj_{N \cap \omega_1}(s_1)$.
\item $t_{n-1} = \proj_{N'\cap \omega_1}(s)$. Then $t_n = \proj_{N\cap \omega_1}(s)$.
\item $t_{n-1}$ is a branch from $N \cap [T]$ that goes through $w \in t_{n-1} \setminus N$. Then $\gamma_n < \gamma_{n-1} = N \cap \omega_1 \leq \Lev(w)$ and thus $\proj_{N\cap \omega_1}(w) = t_{n-1}$. 
\end{itemize} 
\item $t_n$ is a new branch in $N \cap [T]$ that goes through $w \in t_n \setminus N$. Again, clearly $w = t_{n-1}$. We would like to show that $\Lev(t_{n-1}) = \gamma_n$. Note that $\proj_{N \cap \omega_1}(t_{n-1})$ must appear before $t_n$, so either it is a member of $B \cup \{t_0,\dots,t_{n-2}\}$ and thus $t_n$ appears before $t_{n-1}$, or that it is simply $t_{n-1}$, as wanted.
\end{itemize}
This concludes the proof for weak monotonicity of the ordinals $\gamma_n$. The proof shows a bit more --- it shows that for a fixed value of $\gamma_n$, the reasons of appearance weakly increase: we first take all the meets, then all the projections and lastly add the branches. In particular, after adding all the branches we only obtain new meets which are of level strictly above $\gamma_n$.

\item We are now ready to show the bound on number of steps with a fixed $\gamma_n$. The number of possible meets is bounded by the number of pairs that have in the beginning. Each projection is obtained from a single element (and not a meet), and finally, for every one of the projections and every one of the models of height $\gamma_n$ we can add a single branch.\footnote{If a meet adds a branch, then there is a projection that adds the same branch.}

\item In order to bound the number of steps between height of models, note that they must be obtained by meets of pair of elements, as projections and adding branches are done only at level of occurrence which is of the form $N \cap \omega_1$, $N \in \vec{M}_p$. 
\end{enumerate}

Finally, this implies that the whole process is finite, as there are finitely many steps between the heights of models and there are finitely many models in the model sequence.  

\end{proof}
\begin{remark}
The bound for the number of steps that we obtained is extremely loose --- it is super-exponential in the size of $\vec M_p$. By a finer analysis we can show that it is at most linear in the product of $|B|$ and $|\vec{M}_p|$.
\end{remark}

Let us now proceed to the definition of $f_r$. The domain of $f_r$ is going to be $B \cup \{t_i \mid i < n_*\}$. Let $f_r \restriction B = f_p$ and by recursion on $i < n_*$, let us define:

\begin{itemize}
\item If $t_i \notin [T]$, then \[f_r(t_i) = \min\{f_r(s) \mid s\in B \cup \{t_0,\dots, t_{i-1}\},\, t_i \leq_T s\},\]
\item and if $t_i \in [T]$, then \[f_r(t_i) = \max\{f_r(s) \mid s \in t_i\}.\]

\end{itemize} 
We need to show that $f_r$ satisfies the requirements of the definition of the forcing. Let $\tilde{B} = B \cup \{t_0,\dots, t_{n_* - 1}\}$. Let $f_r^i = f_r \restriction B \cup \{t_0, \dots, t_{i - 1}\}$. We will show by induction on $i$ that $f_r^i$ satisfies all requirements, except for closure under branches as in Requirement \ref{condition:2}.

\begin{enumerate}
\item \emph{Weak Monotonicity}. For $i = 0$, it follows from the fact that $p$ is a condition. For successor $i$, if $f_r^i$ would fail to be weakly monotonic, then there is $s\in \dom f_r^{i-1}$ such that monotonicity fails between $f_r(t_{i-1})$ and $f_r(s)$. This clearly cannot happen if $t_i\in [T]$, so let us assume that $t_i\notin [T]$. 

If $t_{i-1} \leq s$ then $f_r^{i-1}(s)$ is one of the elements in the set defining the value of $f_r(t_{i-1})$. If $s \leq_T t_{i-1}$ then there is some $w\in \dom f_r^{i-1}$ such that $t_{i-1} \leq_T w$ and $f_r^{i}(t_{i-1}) = f_r^{i-1}(w)$. Thus, weak monotonicity already fails for $f_r^{i-1}$ for the pair $s,w$.
\item \emph{Closure under branches.} Clearly $\tilde{B}$ satisfies the assertion that if $w \in \tilde{B}$ and there is $b \in N \cap [T]$ such that $N \in \vec{M}_p$, $w\in b \setminus N$ then $b\in \tilde{B}$. So, we show by induction that in this case, $f_r(b) = f_r(w)$. 

First, let us note that without loss of generality $w = \proj_{N\cap\omega_1}(b)$. Indeed, by the order in which elements of $\tilde{B}$ are introduced, unless $b$ is already in $N$ it is introduced only after we introduce its projection to $N' \cap \omega_1$ where $N'$ is a model of minimal value of $N'\cap\omega_1$ that contains $b$ (as in any case in which $b$ can be introduced, we can introduce its projection and the projection is prioritized). If $w \in B$ then clearly $b\in B$ and there is nothing to show. Otherwise, $w$ is a new element of $b$ and the least such element of level $\geq N \cap \omega_1$ which is added is $\proj_{N\cap \omega_1}(b)$. 

{\bf Case 0:} Let us assume first that $t_i = b$ and $w \in B \cup \{t_0,\dots, t_{i-1}\}$. By the order that we introduce the new elements of $\tilde{B}$, $b \in N'$ for some $N' \in \vec{M}_p$ such that $N' \cap \omega_1 < N \cap \omega_1 = \Lev(w)$. 
Let us verify that in this case $f_r(w) = f_r(b)$. In other words, we must verify that there are no larger elements in $b$ which are enumerated before step $i$. By the order of introduction of elements to $\tilde{B}$, this is immediate. 

{\bf Case 1:} $b \in B \cup \{t_0,\dots, t_{i-1}\}$ and $w=t_i$.
Indeed, $f_r(w)$ is defined to be $\min \{f_r(u) \mid t \in B \cup \{t_0,\dots, t_{i-1}\},\, w \leq_T u\}$, so $f_r(w) \leq f_r(b)$ always holds. The elements that appear in this list are either members of $B$ or new branches $c$ that appeared before $w$. Let $c = t_j\in [T]$, $j < i$ such that $w\in c$ and $f_r(c) < f_r(b)$. By the order of inclusion of elements into $\tilde{B}$, $c \in N''$ for some $N'' \cap \omega_1 < N \cap \omega_1$. 

Let $w_0=\proj_{N'' \cap \omega_1}(c) \in b$, $w_1=\proj_{N'\cap \omega_1}(b) \in c$. Then $w_0, w_1 \leq_T w$, and both are introduced before step $i$. Now, we derive a contradiction: if $c$ is introduced before $b$, then the pair $(c, w_1)$ is a smaller counterexample then $(b, w)$. If $b$ is introduced before $c$ then $(b, w_0)$ is a smaller counterexample.      
\item \emph{Closure under Meets and Root} This is trivial.
\end{enumerate}
\end{proof}

	Since $E$ is dense, we may assume, without loss of generality, that $D \subseteq E$ (by picking the elements of $E$ that lie below elements from $D$).
	
    Denote $d^M =\langle  f_d\restriction M,\ \vec{M}_d\cap M\rangle$.

By Lemma \ref{lem:why_d_M_is_a_condition}, if $M\in \vec{M}_d$ then $\vec{M}_d\cap M = res_M\left( \vec{M}_d \right)$, making $d^M$ a condition in $\mathbb{S}$ as well. 

Assume that $\left| f_d \right|= n$ and let $\bar{\mu}=\{\mu_0,\dots, \mu_{k-1}\} = \{N \cap \omega_1 \mid N \in \vec{M}_d\} \cap M$. Consider the set of conditions

\begin{equation*}
    \mathfrak{D} = \left\{ r\in D\ \Big|\ r\leq d^M\wedge \left| f_r \right|= n,\, \{N \cap \omega_1 \mid N \in \vec{M}_r\} = \bar{\mu} \right\}
\end{equation*}

Notice that $\mathfrak{D}$ is definable in $H(\omega_3)$ using parameters from $M$ hence $\mathfrak{D}\in M$. Moreover, as $d\in\mathfrak{D}$, $\mathfrak{D}\neq \emptyset$. 
\footnote{The assumption that $M\in \vec{M}_d$ is important, since then $d^M$ is a condition (Lemma \ref{lem:why_d_M_is_a_condition}) and $d\in \mathfrak{D}$.}

Let us first verify that $f_d \notin M$. Indeed, otherwise any condition $r\in \mathfrak{D} \cap M$ is compatible with $d$. Let $r = \langle f_r, \vec{M}_r\rangle$, $d = \langle f_d, \vec{M}_d\rangle$. In this case $f_r = f_d$, as $f_{d^M} = f_d$ and the size of $f_r$ does not allow it to have any additional elements. Let $\vec{M}$ be the closure under intersections of $\vec{M}_r \cup \vec{M}_d$. In particular, every new element in $\vec{M}$ can be represented as the intersection of an element from $\vec{M}_d$ and an element of $\vec{M}_r$. In particular, if a cofinal branch $b$ appears in such a countable model then it appears in some model $N \in \vec{M}_r$ and in some other model $N' \in \vec{M}_d$. Clearly, not both of those models are transitive, so without loss of generality $N \cap \omega_1 \leq N' \cap \omega_1$ and thus $(N \cap N') \cap \omega_1 = N \cap \omega_1$. If there is some $w\in \dom f_d$ such that $w\in b \setminus (N\cap N')$ then already $b \in \dom f_d$, using the closure under branches of $r$ itself.   

\begin{corollary}
    The set $\bigcup\left\{\dom f_r\cap T \ \Big|\ r\in\mathfrak{D}  \right\}$ is unbounded in $T$.
\end{corollary}


    Since $\forall r\in M\cap D,\ r\perp d$ the same holds for all conditions in $\mathfrak{D}\cap M$. Since $M\in \vec{M}_d$ and $\forall r\in \mathfrak{D}\cap M,\ r\leq d^M$, it follows $\vec{M}_r \leq res_M\left( \vec{M}_d \right)$ hence by lemma \ref{lem:q_leq_res} $\vec{M}_r \parallel \vec{M}_d$.

    Thus, every condition in $\mathfrak{D}$ has a compatible model sequence part to the model sequence of $d$, this means that the incompatibility cannot come from the model sequence.\footnote{As explained above, the new elements which are originated as intersections of elements the model sequences cannot contradict the requirement of closure under branches.}

    So, what can cause this incompatibility? The following lemma enumerates the possible causes for incompatibility.
\begin{lemma}
    \label{lem:reasons_for_incompatibility}

    Let $p\in \mathfrak{D}$

    \begin{enumerate}
        \item \label{incompatibility:1} Incompatibility of monotonicity:
              \begin{equation*}
                  \exists x \in \dom f_p\cap T,\ \exists y\in \dom f_d\cap T,\ x < _T y\wedge f_p\left( x \right) > f_d\left( y \right)
              \end{equation*}
              or
              \begin{equation*}
                  \exists x \in \dom f_p\cap T,\ \exists b\in \dom  f_d \cap \left[ T \right],\ x \in b \wedge f_p\left( x \right) > f_d\left( b \right)
              \end{equation*}
        \item \label{incompatibility:2}$\exists N \in \vec{M}_d$ such that $N\notin M$, $\exists b\in N$, $\exists x_1, x_2\in b\cap\dom f_p\backslash N$ such that $x_1 < x_2$ and $f_p\left( x_1 \right) < f_p\left( x_2 \right)$.\footnote{Here the incompatibility comes from a model $N$ unknown to $M$, such that $N\cap\omega_1 < M\cap\omega_1$, and $\dom f_d$ has a branch from $N$, this branch limits the possible values of $x_1,x_2$ (condition (\ref{condition:2})), but being unknown to $M$, $f_p$ sets $f_p\left( x_1 \right) < f_p\left( x_2 \right)$.}

        \item \label{incompatibility:3}$\exists N\in \vec{M}_d$ such that $N\notin M$, $\exists b_\star \in N\cap\dom f_d\cap\left[ T \right]$, $\exists x\in b_\star\cap\dom f_p\backslash N$ such that $f_p\left( x \right) < f_d\left( b_\star \right)$.\footnote{Here the incompatibility comes from a model $N$ unknown to $M$, such that $N\cap\omega_1 < M\cap\omega_1$ and $\dom f_d$ has a branch from $N$. This branch sets the value of some $x_1\in \dom f_p$ but being unknown to this $f_p$ it set its value to be smaller the $f_d\left( b_\star \right)$.} 
    \end{enumerate}

We say that a condition $p$ has \emph{incompatibility of type \ref{incompatibility:1}} if requirement \ref{incompatibility:1} holds, and similarly for requirements \ref{incompatibility:2} or \ref{incompatibility:3}.  
\end{lemma}
\begin{proof}
First, let us rule out some other possible incompatibilities.
\begin{itemize}
\item  Notice that switching the direction of the inequality in \ref{incompatibility:1}, (i.e., with the roles of $p,d$ reversed), always leads to compatibility. Indeed, for a pair of elements $x \in \dom f_p$, $y\in \dom f_d$, if $x, y \in T$ and $y\notin \dom f_p$ then necessarily $x <_T y$. Moreover, if $b\in \dom f_p$ then $b\in M$ and if $x\in \dom f_d\backslash M$ from condition (\ref{condition:2}), $b\in \dom f_d$ hence in $\dom f_{d^M}$ and since $f_p \leq f_{d^M}$ it follows that $f_d\left( x \right) = f_d\left( b \right) = f_p\left( b \right)$.
\item In \ref{incompatibility:2}, we do not need to worry about $N \in M$, as for those models will appear in the model sequence of $d^M$. 
\item In \ref{incompatibility:3}, the case where the '$<$' changes to '$>$' is covered in (\ref{incompatibility:1}).
\end{itemize}

This lemma is non-trivial, as a condition $q$ must satisfy closure under meets and under adding branches. So we show that if all three possibilities fail, then there is a condition $q$ stronger than both $p$ and $d$. 

Indeed, we know that the model sequences can be combined by adding the required intersections. Let $\vec{M}_q$ be the obtained model sequence. Note that for every $N \in \vec{M}_q$, there is $N' \in \vec{M}_r \cup \vec{M}_d$ such that $N \cap \omega_1 = N' \cap \omega_1$. 

Next, we would like to construct the function $f_q$, extending $f_r \cup f_d$. We will use the following claim.
\begin{lemma}
Let $\vec{M}$ be a condition in the pure side condition forcing. Let $f$ be a finite function from $T \cup [T]$ to $\omega$ such that:
\begin{enumerate}
\item $f$ is weakly monotonic.
\item $\dom f$ is closed under projections to $N \cap \omega_1$ for $N \in \vec{M}$.
\item For every $N \in \vec{M}$, $b \in N \cap [T]$ there are {\bf no} $x_1, x_2 \in \dom f \cap \big((b \setminus N) \cup \{b\}\big)$ such that $f(x_1) < f(x_2)$. 
\end{enumerate}
Then, there is a function $f_*$ extending $f$ such that $\langle f_*, \vec{M}\rangle$ is a condition.
\end{lemma}
\begin{proof}
This lemma resembles Lemma \ref{lemma:closure-under-branches-meets-and-projections}, but differs from it as the starting point of Lemma \ref{lemma:closure-under-branches-meets-and-projections} is a function with domain closed under meets and we are adding the projections and the branches, while here we start with a function with domain closed under certain projections, and close it under meets and the added branches.

 Let us construct recursively a sequence of functions $f_i$ by starting with $f_0 = f$ and repeatedly closing the domain under meets, projections to levels of models and adding branches. We will show, as in Lemma \ref{lemma:closure-under-branches-meets-and-projections}, that this process terminates after finitely many steps and produces a condition. The order in which we introduce new elements is the same as in Lemma \ref{lemma:closure-under-branches-meets-and-projections}.

We will maintain the following inductive hypothesis:
\begin{itemize}
\item $f_i$ is weakly monotone.
\item there is no $b \in N \in \vec{M}$ and $x, y \in \dom f_i$ such that $x\in b\setminus N$ $y \in b \setminus N$ or $y=b$ and $f_i(x) \neq f_i(y)$. 
\end{itemize} 


For each $i$ we add to the domain of $f_{i}$ the element $t_i$ which is either a branch $b$ such that there is $N \in \vec{M}_r$ with $b \in N$, $x \in b \setminus N$, $x \in \dom f_i$, and then we define 
\[f_{i+1}(t_i) = \max(\{f_i(w) \mid w \leq_T t_i,\, w \in \dom f_i\},\]
or $t_i$ an element of the form $s \wedge t$ for $s, t\in \dom f_i$ or the form $\proj_{N\cap \omega_1}(s)$ for $N \in \vec{M}_q$, $s\in \dom f_i$. In this case we define 
\[f_{i+1}(t_i) = \min(\{f_i(w) \mid t_i \leq_T w\,\, w\in \dom f_i\}).\] 
By Lemma \ref{lemma:closure-under-branches-meets-and-projections}, such a process must converge after finitely many steps, and maintains the weak monotonicity.

Let us prove by induction on $i$ that the inductive hypothesis holds. The argument for the preservation of closure under branches is very similar to the proof of Lemma \ref{lemma:closure-under-branches-meets-and-projections}. Let us go through the different possibilities:
\begin{itemize}
\item $t_i$ is a new branch, $b$ add by a model $N \in \vec{M}$ and there is $w\in \dom f_i$ such that $w\in b$. We need to show that $f_i(w) = f_i(b)$. Indeed, before we introduce $b$ we must introduce $u = \proj_{N\cap \omega_1}(b)$. By the inductive hypothesis, $f_i(u) = f_i(w)$, and thus, this is the value of $b$.

\item $t_i$ is a new element, introduced as the meet of a pair of elements, and there is a branch $b \in N$ such that $t_i \in b \setminus N$. Let us assume that there is also $w \in (b \cup \{b\}) \cap \dom f_i$, $t_i < w$. 
By the order in which new elements are introduced, $w$ must be in $\dom f_i$. As $\dom f$ is closed under projections to $N \cap \omega_1$, $\proj_{N\cap \omega_1}(b) = \proj_{N\cap \omega_1}(w) \in \dom f$ and thus in $\dom f_i$. By the inductive hypothesis, $f_i(w) = f_i(\proj_{N\cap \omega_1}(b))$. Since $t_i$ is new, $\proj_{N\cap\omega_1}(t_i) < t_i$ and thus $b$ is introduced before $t_i$. 
So, by the inductive hypothesis, $f_i(w) = f_i(b)$. Moreover, there is no element $z$ above $t_i$ such that $f_i(z) < f_i(b)$, as this would contradict weak monotonicity. So, $f_i(t_i) = f_i(b)$, as needed.
\item $t_i$ is the projection to $N \cap \omega_1$ of some element. Clearly, $t_i$ must be the projection of a certain branch $c \in \dom f_i \setminus \dom f$, as otherwise it would not be new. As $c$ is introduced before $t_i$, it must belong to a model $N'$ with $N'\cap\omega_1 < N \cap \omega_1$. So, there is some $u < t_i$ in $\dom f_i$, $u = \proj_{N' \cap \omega_1}(c)$ and $f_i(u) = f_i(c) = f_{i+1}(t_i)$. Let us assume that there is a branch $b \in N''$ such that $t_i \in b \setminus N''$, and either $f_i(b) \neq f_i(u)$ or there is $w\in b \setminus N''$ with $f_i(w) \neq f_{i+1}(t_i)$. We split into cases:
\begin{itemize}
\item $b \in \dom f_i$. In this case, if $f_i(b) \neq f_i(u)$ then by monotonicity, $f_i(b) > f_i(u)$. If $\proj_{N'' \cap\omega_1}(b) < t_i$, then by the inductive hypothesis, $f_i(\proj_{N'' \cap\omega_1}(b)) = f_i(u)$, (as it is in $c$ as well), but then $f_i(b) = f_i(u)$. Otherwise, as $b\in \dom f_i$, we conclude that $b \in \dom f$, by the order of introduction of elements. But then, by the closure of $f$ under projections, $t_i = \proj_{N'' \cap\omega_1}(b) \in \dom f$ to begin with. 
\item $b\notin \dom f_i$ and there is $w\in b \setminus N''$, $f_i(w) \neq f_{i+1}(t_i)$. Again, if $w > t_i$ then it must be from $\dom f$ and then $t_i = \proj_{N \cap \omega_1}(w) \in \dom f$. So $w < t_i$, and in particular $w\in c$. We again obtain that $f_i(w) = f_i(c)$, as wanted.  
\end{itemize}

\end{itemize}
\end{proof}

In order to apply the lemma, we need to verify that $f = f_r \cup f_d$ satisfies the assumptions of the lemma.

Indeed, the requirements in the statement of the lemma indeed cover all cases with respect to models for $\vec{M}_d$, as the symmetric case cannot occur: if $N \in \vec{M}_q \setminus \vec{M}_d$ if $x \in \dom f_d \setminus N$ and there is $b \in N$ such that $x \in N$ then $N \subseteq M$ and thus $b \in M$ and thus $b \in \dom f_d$ and thus $b \in \dom f_{d^M}$. 



\end{proof}
Let $\left|\mathfrak{D}\right|=\zeta$. We want to fix an enumeration which is absolute between $H(\lambda)$ and $M$.

Since $M \models \left|\mathfrak{D}\right| = \zeta $, $\exists f \in M,\ f : \zeta \rightarrow \mathfrak{D}$ a bijection. Therefore, $H(\lambda)\models f:\zeta\rightarrow \mathfrak{D}$ is a bijection. This bijection now fixes an absolute enumeration of $\mathfrak{D}$.

Denote $f_\alpha = f_{p_\alpha}$ and consider $\mathfrak{E} = \left\{ a^\alpha\ \Big|\ \alpha < \zeta \right\}$ where $a^\alpha$ is a choice of enumerations of $\dom f_{\alpha}$ for $p_\alpha\in\mathfrak{D}$ $\left( \dom f_p = \left\{ a^\alpha_0, \dots a^\alpha_{n-1} \right\} \right)$.

\begin{remark}
    The enumerations in $\mathfrak{E}\cap M$ are absolute between $H(\lambda)$ and $M$, since $M$ can describe $a^\alpha$ explicitly using a first order formula.
\end{remark}

\begin{claim}\label{claim:coloring}
    There exists a coloring (not-disjoint!) of $\zeta$, with the colors $P_{ij}$, $P_{Nk_0k_1}$ and $P_{Nkb_\star}$ for $i,j,k, k_0,k_1 < n$, $N\in \vec{M}_d\wedge (N\cap \omega_1) < (M\cap\omega_1)$ and $b_\star \in \dom f_d\cap\left[ T \right]$ with the following properties:

    \begin{equation*}
        \left(\alpha, \beta\in P_{ij}\rightarrow a^\alpha_i\parallel_T a^\beta_i\right)\wedge \left( \forall \alpha \in P_{ij},\ f_\alpha\left( a^\alpha_i \right) > f_d\left( d_j \right) \right)
    \end{equation*}
    \begin{equation*}
        \begin{matrix}\left(\alpha,\beta\in P_{Nkb_\star}\rightarrow a^\alpha_{k}\parallel a^\beta_{k}\right) & \wedge \\ \left( \forall\alpha\in P_{Nkb_\star}, \level( a^\alpha_k ) \geq N\cap\omega_1 \wedge f_\alpha ( a^\alpha_k) < f_d( b_\star )\right) &\end{matrix}
    \end{equation*}
    and
    \begin{multline*}
        \left(\alpha, \beta \in P_{Nk_0k_1}\rightarrow \pi_{N\cap\omega_1}\left( a^\alpha_{k_1}\right)\neq\pi_{N\cap\omega_1}\left( a^\beta_{k_1} \right) \vee a^\alpha_{k_1} \parallel a^\beta_{k_1}\right) \wedge \\ \left( \forall \alpha \in P_{Nk_0k_1},\ \level(a^\alpha_{k_0}) \geq N\cap\omega_1\wedge a^\alpha_{k_0} < a^\alpha_{k_1} \wedge f_\alpha\left( a^\alpha_{k_0}\right) < f_\alpha\left( a^\alpha_{k_1} \right) \right) 
    \end{multline*}
\end{claim}

\begin{remark}
    Each color is supposed to capture a reason for an incompatibility with $d$. $P_{ij}$ captures reason \ref{incompatibility:1}, $P_{Nk_0k_1}$ captures reason \ref{incompatibility:2} and $P_{Nkb_\star}$ captures reason \ref{incompatibility:3}. 
    
Technically speaking, some of the \emph{names} of the colors are not members of $M$. This is a mere convention in the presentation of the proof. Formally, they should be replaced by natural numbers.    
\end{remark}

\begin{proof}
    Let $X\subset \zeta\cap M$ be a finite set. Consider the sets (defined in $H\left( \omega_3 \right)$).

    \[\begin{matrix}
        P_{ij}^X &= &\left\{ \alpha \in X\ \Big|\ a^\alpha_i < d_j \wedge f_\alpha\left( a^\alpha_i \right) > f_d\left( d_j \right)\right\} \\
        P_{Nk_0k_1}^X &= &\left\{ \alpha\in X\ \Big|\ \psi_1\left( \alpha, N\cap\omega_1, k_0, k_1 \right)\wedge\exists b\in \left[ T \right]\cap\left( N\backslash M \right), a^\alpha_{k_1}\in b \right\} \\
        P_{Nkb_\star}^X & = & \left\{ \alpha\in X\ \Big|\ \psi_2\left( \alpha, N\cap\omega_1, k, b_\star  \right)\wedge a^\alpha_k\in b_\star \right\} \end{matrix}\]
    where
\[\begin{matrix}
        \psi_1\left( \alpha, \beta, k_0, k_1 \right) & = & \level(a^\alpha_{k_0}) \geq \beta& \wedge & a^\alpha_{k_0} < a^\alpha_{k_1}& \wedge & f_\alpha(a^\alpha_{k_0}) < f_\alpha (a^\alpha_{k_1}) \\
        \psi_2\left( \alpha, \beta, k, b_\star \right) & = & \level(a^\alpha_k)\geq \beta &\wedge & f_\alpha(a^\alpha_k) < f_d(b_\star) & & \end{matrix}\]
        and
$i,j,k, k_0,k_1 < n$, $N\in \vec{M}_d$ such that $N\cap \omega_1 < M\cap\omega_1$, and $b_\star \in \dom f_d\cap\left[ T \right]$.

\begin{lemma}
    $P_{ij}^X,\ P_{Nk_0k_1}^X,\ P^X_{Nkb_\star}$ satisfy the properties listed in claim \ref{claim:coloring}.
\end{lemma}

\begin{proof}
    Let $\alpha,\beta \in X$.

    Assume $\alpha,\beta \in P_{ij}^X$ then $a^\alpha_i \parallel a^\beta_i$ since $a^\alpha_i,a^\beta_i \leq_T d_j$.

    Assume $\alpha,\beta \in P_{Nk_0k_1}^X$, if $\pi_{N\cap\omega_1}( a^\alpha_{k_1}) = \pi_{N\cap\omega_1}( a^\beta_{k_1} )$ then $\exists b_1,b_2\in \left[ T \right]\cap N\backslash M$ such that $a^\alpha_{k_1}\in b_1\wedge a^\beta_{k_1}\in b_2$. 
    
    We claim that $N \models b_1 = b_2$. Otherwise, $\exists \alpha < \omega_1^N$ such that $b_1\cap T_\alpha\neq b_2\cap T_\alpha$ which is a contradiction to  $\pi_{N\cap\omega_1}( a^\alpha_{k_1}) = \pi_{N\cap\omega_1}( a^\beta_{k_1} )$. Since $N\models b_1 = b_2$ and $N\prec H(\lambda)$ then $H(\lambda)\models b_1 = b_2$. Hence, $a^\alpha_{k_1}$ and $a^\beta_{k_1}$ are on the same branch i.e.\ $a^\alpha_{k_1}\parallel a^\beta_{k_1}$.

    Assume $\alpha,\beta \in P_{Nkb_\star}^X$, then $a^\alpha_k,a^\beta_k\in b_\star$ hence $a^\alpha_k \parallel a^\beta_k$.
\end{proof}

\begin{lemma}
    $X = \bigcup P_{ij}^X\cup\bigcup P_{Nk_0k_1}^X\cup\bigcup P_{Nkb_\star}^X$.
\end{lemma}

\begin{proof}
    Let $\alpha\in X$, since $X\subset \zeta\cap M,\ p_\alpha\perp d$, from one (or more) of the reasons \ref{incompatibility:1}, \ref{incompatibility:2} and \ref{incompatibility:3} for incompatibility of a condition from Lemma \ref{lem:reasons_for_incompatibility}.

    If the incompatibility is of type \ref{incompatibility:1} then there exists some index $i < n$ and $j < n$ such that $a^\alpha_i < d_j$ but $f_\alpha\left( a^\alpha_i \right) < f_d\left( d_j \right)$. In this case by definition $\alpha\in P_{ij}^X$.

    If the incompatibility is of type \ref{incompatibility:2} then $\exists N\in \vec{M}_d$ such that $N\notin M$ with some $b\in N$ and $x_1,x_2\in b\cap \dom f_\alpha\backslash N$ such that $x_1 < x_2$ and $f_\alpha\left( x_1 \right) < f_\alpha\left( x_2 \right)$. Since $\dom f_\alpha\cap T\backslash N \neq \emptyset$ and $p_\alpha\in M$ then $M\cap\omega_1 > N\cap\omega_1$ (by Corollary \ref{cor:M_cap_omega1_geq}). Let $k_0,k_1$ be the indices of $a^\alpha$ such that $a^\alpha_{k_0} = x_1\wedge a^\alpha_{k_1} = x_2$, then $\alpha \in P_{Nk_0k_1}^X$.

    If the incompatibility is from \ref{incompatibility:3} then $\exists N\in \vec{M}_d$ s.t. $N\notin M$ with some $b_\star \in N$ and $x_1\in b_\star\cap\dom f_\alpha \backslash N$ and $f_\alpha\left( x_1 \right) < f_d\left( b_\star \right)$. Let $k$ be the index such that $a^\alpha_{k} = x_1$ then $\alpha\in P_{Nkb_\star}^X$.

    Thus, $X = \bigcup P_{ij}^X\cup\bigcup P_{Nk_0k_1}^X\cup\bigcup P_{Nkb_\star}^X$.
\end{proof}

    We showed that we can construct a coloring in $H(\lambda)$ of every finite $X\subset \zeta\cap M$, satisfying the properties listed in claim \ref{claim:coloring}.

    Let $X\subset \zeta\cap M$, so
\[H(\lambda) \models \mathrm{There\ is\ a\ coloring\ of\ } X\ \mathrm{with\ the\ properties\ listed\ in\ claim\ \ref{claim:coloring}}.\]
    Recall that there is $M'\prec H(\lambda)$ such that $M = M'\cap H(\omega_2)$. So, 
\[M' \models \mathrm{There\ is\ a\ coloring\ of\ } X\ \mathrm{with\ the\ properties\ listed\ in\ claim\ \ref{claim:coloring}}.\]

    Let $P_{ij}^{X,M},P_{Nk_0k_1}^{X,M},P_{Nkb_\star}^{X,M}\in M$ be this coloring of $X$ inside $M$. Next we will use the compactness theorem to get a coloring of all of $\zeta$.

\begin{remark}
    \label{rem:why_in_M}
    For the reflection argument above to work every variable in the definition of the properties in claim \ref{claim:coloring} needs to be in $M$.

    Notice that in $P_{ij}$ the definition uses $f_d(d_j)\in\omega$ which is inside $M$. In $P_{Nk_0k_1}$ the definition uses $N\cap\omega_1$ which is in $M$ since $N\cap\omega_1 <M\cap\omega_1$. And in $P_{Nkb_\star}$ the definition uses $N\cap\omega_1$ and $f_d(b_\star)$ in the same manner as in $P_{ij}$.

As we remarked before, the mentions of $N$ and $b_\star$ as names for the colors are merely place holders.
\end{remark}

    Add $\zeta$ constant symbols $c_\alpha$ to our language $(\in)$ and another finite amount of constants $\widetilde{P}_{ij},\widetilde{P}_{Nk_0k_1},\widetilde{P}_{Nkb_\star}$ (or equivalently, unary predicates).
    Let us denote by $\omega_1^N = N \cap \omega_1$. Define the following theories in $M$
\[
\begin{matrix}
        \mathfrak{T}_0 &= &\{ c_\alpha\notin \widetilde{P}_{ij} \vee c_\beta\notin \widetilde{P}_{ij} & \mid & \alpha,\beta < \zeta,\ i, j< n,\ a^\alpha_i\perp a^\beta_i\} \\
        
        \mathfrak{T}_1 &= & \{ c_\alpha \notin \widetilde{P}_{ij} & \mid& \alpha <\zeta,\ i,j < n,\ f_\alpha(a^\alpha_i) < f_d(d_j) \} \\ 
        
        \mathfrak{T}_2 & = & \{ c_\alpha\notin \widetilde{P}_{Nk_0k_1} \vee c_\beta\notin \widetilde{P}_{Nk_0k_1}& \mid & \alpha <\zeta,\\ & & & & \pi_{\omega_1^N}\left( a^\alpha_{k_1}\right) = \pi_{\omega_1^N}\left( a^\beta_{k_1} \right)  \wedge a^\alpha_{k_1}\perp a^\beta_{k_1} \} \\
        
        \mathfrak{T}_3 &= & \{ c_\alpha \notin \widetilde{P}_{Nk_0k_1}& \mid& \alpha <\zeta, \\ & & & & \level(a^\alpha_{k_0}) < \omega^N_1\vee a^\alpha_{k_0} < a^\alpha_{k_1}\vee f_\alpha(a^\alpha_{k_0})\geq f_\alpha(a^\alpha_{k_1})  \} \\
        
        \mathfrak{T}_4 &= & \{ c_\alpha\notin \widetilde{P}_{Nkb_\star} \vee c_\beta\notin \widetilde{P}_{Nkb_\star}& \mid & \alpha < \zeta,\ \\ & & & & \pi_{\omega^N_1}(a^\alpha_{k}) = \pi_{\omega^N_1}(a^\beta_k)\wedge a^\alpha_k \perp a^\beta_k  \} \\
        
        \mathfrak{T}_5 &= & \{ c_\alpha\notin \widetilde{P}_{Nkb_\star}& \mid & \alpha  <\zeta, \\ & & & & \level(a^\alpha_k) < \omega_1^N\vee f_\alpha(a^\alpha_k)\geq f_d(b_\star)  \} \\
    \end{matrix}\]
and 
\[\begin{matrix}        \mathfrak{T}_6 &= & \{ \bigvee_{ij}c_\alpha \in \widetilde{P}_{ij}\vee\bigvee_{Nk_0k_1}c_\alpha \in \widetilde{P}_{Nk_0k_1} \vee \bigvee_{Nkb\star}c_\alpha\in\widetilde{P}_{Nkb_\star}\ & \mid & \alpha <\zeta \} 
\end{matrix}\]
    Consider $\mathfrak{T} = \bigcup_i\mathfrak{T}_i$.

\begin{lemma}
    $M\models \mathfrak{T}$ is finitely satisfiable.
\end{lemma}

\begin{proof}
    Since $\mathfrak{T}$ is definable in $M$ using parameters from $M$ (see Remark \ref{rem:why_in_M}), and $M$ is a model of enough set theory $\mathfrak{T}\in M$.
    Work in $M$ and let $\mathcal{X}\subset\mathfrak{T}$ be finite. Let $X = \left\{ \alpha\ \Big|\ c_\alpha \mathrm{\ is\ in\ a\ formula\ of\ }\mathcal{X} \right\}$. Interpret $c_\alpha$ as $\alpha$ and $\widetilde{P}_{ij} = P_{ij}^{X,M},\ \widetilde{P}_{Nk_0k_1} = P_{Nk_0k_1}^{X,M},\ \widetilde{P}_{Nkb_\star} = P_{Nkb_\star}^{X,M}$. By the previous argument this satisfies $\mathcal{X}$. Since this is true for every finite $\mathcal{X}\subset\mathfrak{T}$, $\mathfrak{T}$ is finitely satisfiable.
\end{proof}

    Since $\mathfrak{T}$ is finitely satisfiable, by the compactness theorem, it is satisfiable. Let $P_{ij},\ P_{Nk_0k_1},\ P_{Nkb_\star}$ be the sets which satisfy it. Thus, $P_{ij},\ P_{Nk_0k_1},\ P_{Nkb_\star}$ color all of $\zeta$ in $M$ i.e.
    \begin{equation*}
        M\models P_{ij},P_{Nk_0k_1},P_{Nkb_\star}\  is\ a\ coloring\ of\ \zeta\ with\ the\ properties\ listed\ in\ \ref{claim:coloring}
    \end{equation*}
    hence
    \begin{equation*}
        H(\lambda)\models P_{ij},P_{Nk_0k_1},P_{Nkb_\star}\  is\ a\ coloring\ of\ \zeta\ with\ the\ properties\ listed\ in\ \ref{claim:coloring}
    \end{equation*}
\end{proof}

Denote by $P_{ij}^\mathfrak{D},P_{Nk_0k_1}^\mathfrak{D},P_{Nkb_\star}^\mathfrak{D}$ the set of conditions in our forcing notion and let  $P^\mathfrak{D} = \left\{ p_\alpha\in \mathfrak{D}\ \Big|\  \alpha\in P\right\}$. Then there exists some color $P$ with $d\in P^\mathfrak{D}$.

\begin{notation}
    For condition $p$, if indices $i, j$ are not the reason for incompatibility with $d$ in the sense of incompatibility reason \ref{incompatibility:1} we write $p \parallel_{(1)_{ij}}d$. Similarly, with incompatibility reasons \ref{incompatibility:2} and \ref{incompatibility:3}.
\end{notation}

\begin{claim}
    \label{claim:parallel_1}
    If $d\in P_{ij}^\mathfrak{D}$ then $\exists X\subset P_{ij}^\mathfrak{D},\ X\in M $ such that $d\in X$ and $\forall p \in X\cap M,\, p\parallel_{(1)_{ij}} d$.
\end{claim}

\begin{proof}
    Suppose $d\in P_{ij}^\mathfrak{D}$ for some $i,j < n$. If $d_i \in M$, take the set

    \begin{equation*}
        X = \left\{ \alpha \in P_{ij}\, \Big|\, a^\alpha_i = d_i \right\}
    \end{equation*}

    Since $f_d(d_i) > f_d(d_j)$ and $d_i,d_j\in \dom f_d$ then $d_i \nleq d_j$ proving every condition in $X$ is $\parallel_{(1)_{ij}}$ to $d$. Since $d_i\in M$, $X$ is definable in $M$ i.e.\ $X\in M$.

    If $d_i \notin M$, then the set $Y = \left\{ a^\alpha_i\ \Big|\ \alpha\in P_{ij} \right\}$ is unbounded since $d_i \in Y\wedge d_i\notin M$ by lemma \ref{lem:X_is_unbounded}. Moreover, $\forall x,y\in Y,\ x\parallel y$. $Y$ then defines a branch via
    \begin{equation*}
        B = \bigcup\left\{ seg(x)\ \Big|\ x\in Y \right\}
    \end{equation*}

    This branch goes through $d_i$. Since $Y$ is definable in $M$, $Y\in M$ and using $Y$ we can build $B$ thus $B\in M$. 
    Then, by condition (\ref{condition:2}), $B\in \dom f_d$ and $f_d(B) = f_d(d_i)$. We consider $d_i\wedge d_j\in \dom f_d$. If $d_i \wedge d_j\notin M$ then we get a contradiction since

    \begin{equation*}
        f_d(B) = f_d(d_i\wedge d_j) \leq f_d(d_j) < f_d(d_i) = f_d(B)
    \end{equation*}

    If $d_i\wedge d_j\in M$ then the set

    \begin{equation*}
        X = \left\{ p\in P_{ij}^\mathfrak{D}\ \Big|\, \forall x\in \dom f_p\backslash \dom f_{d^M},\, \level(x) > \level(d_i\wedge d_j) \right\}
    \end{equation*}

    will satisfy the requirement. First $d\in X$ since $\forall x\in \dom f_d\backslash \dom f_{d^M}$, $\level(x) \geq  M\cap\omega_1$ and since $d_i\wedge d_j\in M,\, \level(d_i\wedge d_j) < M\cap\omega_1$. Let $p\in X\cap M$, then the $i^{th}$ element is above $d_i\wedge d_j$ (by definiton) and below $d_i$ ($d_i \notin M$) hence it cannot have incomptibility of type \ref{incompatibility:1}, for the indices $i,j$.
\end{proof}

\begin{claim}
    \label{claim:parallel_2}
    If $d\in P_{Nk_0k_1}^\mathfrak{D}$ then $\exists X\subset P_{Nk_0k_1}^\mathfrak{D},\ X\in M $ such that $d\in X$ and $\forall p \in X\cap M,\, p \parallel_{(2)_{Nk_0k_1}} d$.
\end{claim}

\begin{proof}
    Suppose $d\in P_{Nk_0k_1}^\mathfrak{D}$ for some $N,k_0,k_1$ and suppose $d_{k_0}\in M$. If $d_{k_1}\notin M$ then

    \begin{equation*}
        Y = \left\{ a^\alpha_{k_1}\, \Big|\, \alpha\in P_{Nk_0k_1},\ \pi_{N\cap\omega_1}(a^\alpha_{k_1}) = \pi_{N\cap\omega_1}(d_{k_1}) \right\}
    \end{equation*}

    defines a branch $B$ going through $d_{k_1}$. Then by closure under branches of $d$, $B\in \dom f_d$. Consider the set

    \begin{equation*}
        X = \left\{ \alpha \in P_{Nk_0k_1}\, \Big|\, a^\alpha_{k_0} = d_{k_0}\right\}
    \end{equation*}

    Then $d\in X$ and taking the subset

    \begin{equation*}
        X' = \left\{ \alpha\in  X\,  \Big|\, \forall x\in \dom f_{p_\alpha}\backslash \dom f_{d^M},\ \level(x) > \level(d_{k_0}) \right\}
    \end{equation*}

    Assume $\exists b\in N\cap\left[ T \right]\cap\dom f_d$ and $\alpha\in X'$ such that $a^{\alpha}_{k_1}\in b$ then $d_{k_0} < a^\alpha_{k_1} < B\wedge b\in \dom f_d$. Notice that there can be at most one branch satisfying this, since the projection on $N\cap\omega_1$ is the same for all branches going through $d_{k_0}$. If $B\wedge b \notin M$ then from condition \ref{condition:2}, $ f_d(d_{k_1} )= f_d(B) = f_d(b) = f_d(d_{k_0})$ which is a contradiction. Hence, $B\wedge b \in M$, consider the set
    \[      X'' = \left\{ \alpha\in  X,  \Big|\, \forall x\in \dom f_{p_\alpha}\backslash \dom f_{d^M},\ \level(x) > \level(B\wedge b) \right\}\]
    By definition $d \in X''$, and for all $b'\in N$ and for all $p_\alpha\in X''$ we have that $a^\alpha_{k_1} \notin b'$. Hence, $X''$ is the required set.

    If $d_{k_1}\in M$, then taking
    \begin{equation*}
        X = \left\{\alpha\in P_{Nk_0k_1}\, \Big|\, a^\alpha_{k_0} = d_{k_0}\wedge a^\alpha_{k_1} = d_{k_1}\right\}.
    \end{equation*}
    gives the required set. If $d_{k_0}\notin M$, we define
    \begin{equation*}
        Y = \left\{ a^\alpha_{k_1}\ \Big|\  \alpha \in P_{Nk_0k_1}\wedge \pi_{N\cap\omega_1}\left( a^\alpha_{k_1} \right) = \pi_{N\cap\omega_1}\left( d_{k_1} \right) \right\}.
    \end{equation*}

    Since $d_{k_1}\in Y$ and $d_{k_1}\notin M$, by lemma \ref{lem:X_is_unbounded} $Y$ is unbounded and defines a branch

    \begin{equation*}
        B = \bigcup\left\{ seg(x)\ \Big|\ x\in Y \right\}
    \end{equation*}

    This branch goes through $d_{k_1}$ therefore also through $d_{k_0}$ and we get a contradiction since

    \begin{equation*}
        f_d(B) = f_d(d_{k_0}) < f_d(d_{k_1}) = f_d(B).
    \end{equation*}
\end{proof}

\begin{claim}
    \label{claim:parallel_3}
    If $d\in P_{Nkb_\star}^\mathfrak{D}$ then $\exists X\subset P_{Nkb_\star}^\mathfrak{D},\ X\in M $ such that  $d\in X$ and $\forall p \in X\cap M,\,  p \parallel_{(3)_{Nkb_\star}} d$.
\end{claim}

\begin{proof}
    Suppose $d\in P_{Nkb_\star}^\mathfrak{D}$ for some $N,k,b_\star$ and suppose $d_k\in M$ consider the set

    \begin{equation*}
        X = \left\{\alpha \in P_{Nkb_\star}\, \Big|\, a^\alpha_k = d_k \right\}.
    \end{equation*}

    By the same arguments this is the required set.

    If $d_k\notin M$, then the set $\left\{ a^\alpha_k\, \Big|\, \alpha \in P_{Nkb_\star} \right\}$ defines a branch $B\in M$ going through $d_k$ and by closure under branches, $B\in \dom f_d$
    \begin{equation*}
        f_d(B) = f_d(d_k) < f_d(b_\star)
    \end{equation*}

    Consider $b_\star\wedge B\in \dom f_d$. If $b_\star\wedge B\notin M$ then

    \begin{equation*}
        f_d(B) = f_d(B\wedge b) = f_d(b)
    \end{equation*}

    Contradiction. Hence, $b_\star\wedge B\in M$ and the set
    \begin{equation*}
        X = \left\{p \in P_{Nkb_\star}^\mathfrak{D}\, \Big|\, \forall x\in \dom f_p\backslash \dom f_{d^M}, x \parallel b_\star\wedge B\rightarrow b_\star\wedge B < x\right\}
    \end{equation*}
    contains $d$. Let $p\in X$ then since every new element of $\dom f_p$ which can cause incompatibility (from reason \ref{incompatibility:3}) is \underline{above} the point where the branches $B$ and $b_\star$ diverge thus cannot be the cause for the incompatibility.
\end{proof}

\begin{claim}
    \label{claim:p_parallel_d}
    $\exists p\in \mathfrak{D}$ such that $p\parallel d$.
\end{claim}

\begin{proof}
    Since

    \begin{equation*}
        \mathfrak{D} = \bigcup P_{ij}^\mathfrak{D}\cup \bigcup  P_{Nk_0k_1}^\mathfrak{D} \cup \bigcup P_{Nkb_\star}^\mathfrak{D}
    \end{equation*}
    $\exists P^\mathfrak{D}$ a color such that $d\in P^\mathfrak{D}$. By claims \ref{claim:parallel_1}, \ref{claim:parallel_2}, \ref{claim:parallel_3} there exists $X\subset P\wedge X\in M$ such that $d\in X$ and $\forall x\in X\cap M$ there is one less incompatibility reason between $x$ and $d$ (depends on the color $P$). Employing this argument repeatedly gives, after a finite number of steps, a set $X\in M$ such that $d\in X$ and $\forall x\in X\cap M,\, x\parallel d$. Since $d\in X$, $M$ thinks $X\neq\emptyset$ thus $\exists p\in X\cap M$.
\end{proof}

    By claim \ref{claim:p_parallel_d}, let $p\in \mathfrak{D}\cap M$ be a condition such that $p\parallel d$. This is a contradiction to the assumption that all conditions in $\mathfrak{D}\cap M$ are incompatible with $d$.
\end{proof}

\begin{theorem}
    $\mathbb{S}$ is proper for $\mathcal{T}$.
\end{theorem}

\begin{proof}
    Even though this case is essentially covered by the proof of Lemma \ref{lemma:P_is_proper} we present it because of its simplicity. Let $W\in \mathcal{T}$, again we go toward contradiction, the difference here is that by a simple re-definition of $\mathfrak{D}$ we can avoid all reasons for incompatibility listed in Lemma \ref{lem:reasons_for_incompatibility}. Define
\[
        \mathfrak{D} =  \left\{ r\in D\, \Big|\, \begin{matrix}
r\leq d^W \wedge |\vec{M}_r| = |\vec{M}_d|\wedge \\
        |f_r| = |f_d|\wedge |\dom f_r\cap[T]| = |\dom f_d\cap[T]|\end{matrix} 
        \right\}
\]

    Since $d\in\mathfrak{D}$ this set is non-empty, and since $d\notin W$ then $\mathfrak{D}\cap W \neq \emptyset$. We claim that all conditions in $D\cap W$ are compatible to $d$, to do that we need to show all reasons for incompatibility are not applicable here. The argument is simple, since $T\subset W$ and $|f_r| = |f_d|\wedge |\dom f_r\cap[T]| = |\dom f_d\cap[T]|$, the conditions in $\mathfrak{D}$ can only add  branches. Now taking any $p\in \mathfrak{D}\cap W$ leads to a contradiction.
\end{proof}

\section{The Generic Object}\label{section:the-generic-object}

Let $G\subset \mathbb{S}$ be a generic filter. Consider the function

\begin{equation*}
    f_G = \bigcup_{ \langle  f,\vec{M}\rangle\in G} f
\end{equation*}

\begin{claim}
    $\range(f) =\omega$.
\end{claim}

\begin{proof}
    Consider the set

    \begin{equation*}
        D_n = \left\{ p\, \big| n\in\range(f_p) \right\},\quad n <\omega
    \end{equation*}

    We show it is dense. Let $q\in \mathbb{S}$ be some condition and let $\pi_1 : T\backslash T_0\rightarrow T_1$ be the projection on the first level of the tree. Since $\left| T_1 \right| = \aleph_0$ we have that $T_1 \backslash \pi_1 \left(\dom f_q\right)\neq\emptyset$. Choose $x\in T_1 \backslash \pi_1 \left(\dom f_q\right) $ then

    \begin{equation*}
        D_n \ni \langle  f_q\cup \langle  x,n \rangle,  \vec{M}_q \rangle = q' \leq q
    \end{equation*}

    Thus, $D_n$ is dense.
\end{proof}

\begin{proposition}
    \label{prop:dom_f_G}
    $\dom f_G = T\cup \left[ T \right]$.
\end{proposition}

\begin{proof}
    Let $x\in T\cup \left[ T \right]$ and consider the set
\[D_x = \left\{ p\, \big|\, x\in \dom f_p \right\}.\]
    Let us show that $D_x$ is dense. Let $q\in \mathbb{S}$ be some condition. We want to add $x$ to the domain of $f_q$ without changing $\vec{M}_q$ such that the new function is consistent with all the requirements of a condition in Definition \ref{def:forcing_poset}. We are going to construct a function $f$ such that $\langle f, \vec{M}_q\rangle$ is a condition in $D_x$ stronger than $q$. 
    
    Without loss of generality, $\dom f_q$ is closed under projection to $N \cap \omega_1$ for $N \in \vec{M}_q$. By the proof of Lemma \ref{lem:reasons_for_incompatibility} it is sufficient to show that there is a function $f_r$ with domain which is also closed under projections to levels $N \cap \omega_1$, for $N \in \vec{M}_q$, such that $x \in \dom f_r$ and $f_r \cup f_q$ is weakly monotonic. 
    
    So let $\dom f_r$ be the set \[\{x\} \cup \{\proj_{N\cap \omega_1}(x) \mid N \in \vec{M}, \Lev(x) \geq N \cap \omega_1\}.\] 
    Let $f_r(y) = \max \{f_q(w) \mid w \leq_T y\}$. As the root belongs to $\dom f_q$, this set is always non-empty. 
    
    We need to show that there is no $N \in \vec{M}_q$ and a branch $b \in N$ such that there are  
    $y \in \dom f_r$, $w\in \dom f_q$, $y, w \in \{b\} \cup (b\setminus N)$ and $f_r(y)\neq f_q(w)$.  
    Indeed, if $w \in b \setminus N$ then $b \in \dom f_q$ and thus for every $u \in (b \setminus N) \cap \dom f_q$, $f_q(u) = f_q(b) = f_q(w)$. Moreover, $\proj_{N\cap\omega_1}(b)\in\dom f_q$. So  as $y \in b \setminus N$, $\proj_{N\cap \omega_1}(b) \leq y$ and thus $f_r(y) = f_r(b)$.
    
	So, we can apply Lemma \ref{lem:reasons_for_incompatibility} and obtain a function $f$ such that $f\supseteq f_q \cup f_r$ and $\langle f, \vec{M}_q\rangle$ is a condition. 
\end{proof}
\begin{remark}\label{remark:arbitrary-values}
Note that if there is no $w \in \dom f_q$ such that $x \leq w$, then we can set the value of $f_r(x)$ to be arbitrarily high. Indeed, as long as we main the closure under projections and the weak monotonicity of the union, the construction goes through.
\end{remark}
\begin{theorem}
    Let $\mathbb{Q}$ be a forcing notion from the ground model and let us assume that  
    \begin{equation*}
        \mathbb{Q}\Vdash \dot{b}\subset \check{T}\mathrm{\, is\, a\, new\, branch}.
    \end{equation*}
    then
    \begin{equation*}
        \mathbb{S}\times \mathbb{Q}\Vdash \left|\check{\omega}_1^V\right| = \aleph_0.
    \end{equation*}
\end{theorem}
\begin{corollary}
As a simple corollary, taking $\mathbb{Q}$ be the tree $T$ with the reserved order, we conclude that the forcing $\mathbb{S}$ forces the tree $T$ to be non-distributive.
\end{corollary}
\begin{proof}
    Let $\mathbb{Q}$ be a forcing such that 
    \begin{equation*}
        \mathbb{Q}\Vdash \dot{b}\subset \check{T}\text{ is a new branch,}
    \end{equation*}
    and let us assume that $\mathbb{Q}$ does not collapse $\omega_1$ in $V$, as otherwise we have nothing to do. 
    Consider the set
    \begin{equation*}
        D_n = \left\{ (p,q)\in \mathbb{S}\times\mathbb{Q}\, \Big|\, q\Vdash \check{x}\in \dot{b}\wedge p\Vdash \dot{f}_G(\check{x}) >\check{n} \right\}.
    \end{equation*}
    Let us show that it is dense in $\mathbb{S}\times \mathbb{Q}$. 
    Let $(p_1,q_1)\in \mathbb{S}\times \mathbb{Q}$ be a condition. 

	Let $B = (\dom f_{p_1} \cap [T]) \cup \bigcup \{N \cap [T] \mid N \in \vec{M}_{p_1}\}$ --- this is a countable collection of branches. 
    Since $q_1$ forces that $\dot{b}$ is a new branch and that $\omega_1$ is still regular, there is an ordinal $\beta < \omega_1$ (depending on the generic filter) such that for every $b \in B$, $\level(\dot{b} \wedge b) < \beta$ and moreover $\beta > \level(x)$ for every $x \in \dom f_{p_1}$. Let $q_2 \leq q_1$ be a condition that decides the value of the ordinal $\beta$ and decide the value of the unique $z \in \dot{b} \cap T_{\beta}$. Extend $q_2$ to a condition $q\in E$ --- the dense set of conditions with domain closed under projections. 
    
   Note that there is no element in $\dom f_q$ which is above $z$, as such an element must be either a branch from one of the models or on one of those branches. 
   
   So, by Remark \ref{remark:arbitrary-values}, we can set the value of $z$ to be any natural number which is larger than $\max \{f_{q}(y) \mid y \leq_T z\}$, and in particular larger than $n$. Let $q'$ be a strengthening of $q$ with $f_{q'}(z) > n$. The obtained condition $\langle p_2, q'\rangle \in D_n$. 

    Finally, let $G\subset \mathbb{S}\times \mathbb{Q}$ be a generic filter and let $f_G$ be the generic function and $b_G$ the generic branch. Since $D_n\cap G\neq \emptyset$ for all $n\in \mathbb{N}$ we have

    \begin{equation*}
        \sup\range(f_G\restriction b_G) = \omega
    \end{equation*}

    $f_G$ is a monotone unbounded function between $b_G$ (which is isomorphic to $\omega_1^V$) and $\omega$. Thus, $\left|\omega_1^V\right| = \aleph_0$.

\end{proof}   

\section{Open Questions}

In this work we have managed to construct a forcing-notion that forces a limitation on the ability to add branches to a Kurepa tree. We showed that no forcing from the ground model is able to add a new branch without destroying the tree. But, there can be new forcing-notions added by $\mathbb{S}$, can they add branches to $T$ without damaging it? When we have tried to prove this for an iteration $\mathbb{S} \ast \dot{\mathbb{Q}}$ instead of a product, we failed since our conditions can not determine the branch without also determining the function $f_G$.

A natural question to ask is

\begin{question}
    \label{que:seal_completely}
    Is there a forcing-notion that seals a Kurepa tree $T$ completely to all forcing notions in $V[G]$?
\end{question}

If there is a Woodin cardinal, then the stationary tower forcing with critical point $\omega_2$ will always add a new branch without collapsing $\aleph_1$ (even though $\aleph_2$ will be collapsed). So, it seems like a different method of sealing is needed in order to work in $\mathrm{ZFC}$.

If we restrict ourselves to forcing notions of cardinality $\aleph_1$, as in Section \ref{section:distributive-trees}, we might be able to get a simpler result, by iterating the forcing from this paper.
\begin{question}
Is it possible to iterate the forcing notion from Section \ref{section:sealing-trees-forcing}, without collapsing cardinals, and obtain a model in which there are Kurepa trees and no forcing of cardinality $\aleph_1$ can introduce a branch to a Kurepa tree without collapsing $\aleph_1$?
\end{question}
Note that specialization posets such as \cite{GolshaniShelah21} are incompatible with the existences of Kurepa trees.



\begin{thebibliography}{10}

\bibitem{Uri}
Uri Abraham.
\newblock {\em Proper Forcing}, pages 333--394.
\newblock Springer Netherlands, Dordrecht, 2010.

\bibitem{BaumgartnerMalitzReinhardt}
J.~Baumgartner, J.~Malitz, and W.~Reinhardt.
\newblock Embedding trees in the rationals.
\newblock {\em Proc. Nat. Acad. Sci. U.S.A.}, 67:1748--1753, 1970.

\bibitem{CummingsHandbook}
James Cummings.
\newblock Iterated forcing and elementary embeddings.
\newblock In {\em Handbook of set theory. {V}ols. 1, 2, 3}, pages 775--883.
  Springer, Dordrecht, 2010.

\bibitem{CummingsForemanMagidor}
James Cummings, Matthew Foreman, and Menachem Magidor.
\newblock Squares, scales and stationary reflection.
\newblock {\em J. Math. Log.}, 1(1):35--98, 2001.

\bibitem{Kurepa}
Kurepa Duro.
\newblock Ensembles ordonnés et ramifiés.
\newblock {\em Univ. Belgrad}, 4:1--138, 1935.

\bibitem{GolshaniShelah21}
Mohammad Golshani and Saharon Shelah.
\newblock Specializing trees and answer to a question of {W}illiams.
\newblock {\em J. Math. Log.}, 21(1):Paper No. 2050023, 20, 2021.

\bibitem{HayutMuller}
Yair Hayut and Sandra M{\"u}ller.
\newblock Perfect subtree property for weakly compact cardinals.
\newblock {\em Israel Journal of Mathematics}, pages 1--22, 2022.

\bibitem{Jech}
Thomas Jech.
\newblock {\em Set theory}.
\newblock Springer Monographs in Mathematics. Springer-Verlag, Berlin, 2003.
\newblock The third millennium edition, revised and expanded.

\bibitem{JechSolovay}
Thomas~J. Jech.
\newblock Trees.
\newblock {\em J. Symbolic Logic}, 36:1--14, 1971.

\bibitem{Krueger2}
John Krueger.
\newblock Forcing with adequate sets of models as side conditions.
\newblock {\em MLQ Math. Log. Q.}, 63(1-2):124--149, 2017.

\bibitem{Krueger}
John Krueger.
\newblock The approachability ideal without a maximal set.
\newblock {\em Ann. Pure Appl. Logic}, 170(3):297--382, 2019.

\bibitem{Kunen}
Kenneth Kunen.
\newblock {\em Set theory}, volume~34 of {\em Studies in Logic (London)}.
\newblock College Publications, London, 2011.

\bibitem{Mitchell}
William~J Mitchell.
\newblock $i[\omega_2]$ can be the nonstationary ideal on
  $\mathrm{Cof}(\omega_1)$.
\newblock {\em Trans. Amer. Math. Soc}, 361(2):561--601, 2009.

\bibitem{Neeman}
Itay Neeman.
\newblock Forcing with sequences of models of two types.
\newblock {\em Notre Dame Journal of Formal Logic}, 55(2):265 -- 298, 2014.

\bibitem{Neeman2}
Itay Neeman.
\newblock Two applications of finite side conditions at {$\omega_2$}.
\newblock {\em Arch. Math. Logic}, 56(7-8):983--1036, 2017.

\bibitem{Poor}
M\'{a}rk Po\'{o}r and Saharon Shelah.
\newblock Characterizing the spectra of cardinalities of branches of {K}urepa
  trees.
\newblock {\em Pacific J. Math.}, 311(2):423--453, 2021.

\end{thebibliography}

\end{document}